\documentclass{article}
\usepackage{microtype}
\usepackage{comment}
\usepackage{amsmath}
\usepackage{amssymb}
\usepackage{overpic}
\usepackage{tikz}
\usetikzlibrary{arrows.meta,calc}

\usepackage{tabularx,booktabs} 
\usepackage{multirow} 
\usepackage{stmaryrd}
\usepackage{algorithm}
\usepackage{algpseudocode}

\usepackage[dvipsnames]{xcolor}
\usepackage{mathtools}

\usepackage[a4paper]{geometry}
\usepackage[shortlabels]{enumitem}

\usepackage{amsthm}
\usepackage{thmtools}
\usepackage{hyperref}
\usepackage{calrsfs}
\usepackage{hyperref} 
\hypersetup{
	colorlinks = true,
	linkcolor = {cyan},
	citecolor = {cyan},
} 

\usepackage{thm-restate}
\declaretheorem[numberwithin=section]{theorem}
\declaretheorem[numberlike=theorem]{proposition}
\declaretheorem[numberlike=theorem]{corollary}
\declaretheorem[numberlike=theorem]{lemma}

\declaretheorem[numberlike=theorem,style=definition]{example}
\declaretheorem[numberlike=theorem,style=definition]{examples}

\usepackage[indLines=true]{algpseudocodex}

\usepackage{minted}

\NewDocumentCommand{\NN}{}{\mathbb{N}}
\NewDocumentCommand{\ZZ}{}{\mathbb{Z}}

\newcommand{\SL}{\mathrm{SL}}
\newcommand{\Mod}{\mathrm{Mod}}

\makeatletter

\newcommand{\address}[1]{%
	\g@addto@macro\@addresses{\noindent#1\par}%
}
\newcommand{\email}[1]{%
	\g@addto@macro\@addresses{\noindent\texttt{#1}\par\medskip}%
}
\newcommand{\@addresses}{}
\newcommand{\printaddresses}{%
	\section*{}
	\@addresses
}
\makeatother

\title{Exact curve counting of given word length \\ on the once-punctured torus}

\author{Filippo Baroni \and David Fisac \and Mingkun Liu}

\address{\textbf{Filippo Baroni}\\
	Trinity College Dublin, Dublin, Ireland}
\email{baronif@tcd.ie}

\address{\textbf{David Fisac}\\
	Université Paris-Est Créteil, CNRS, LAMA UMR8050, F-94010 Créteil, France}
\email{david.fisac-camara@cnrs.fr}

\address{\textbf{Mingkun Liu}\\
	LAGA, Université Sorbonne Paris Nord, Villetaneuse, France}
\email{mingkun.liu@math.univ-paris13.fr}

\newcommand\blfootnote[1]{%
	\begingroup
	\renewcommand\thefootnote{}\footnote{#1}%
	\addtocounter{footnote}{-1}%
	\endgroup
}

\date{}

\begin{document}
	\maketitle
	
	\begin{abstract}
		On the once-punctured torus, we give an exact formula for the number of curves in any given mapping class group orbit of given word length. This settles a conjecture of Chas in \cite{Chas}. 
	\end{abstract}
	
	\section{Introduction}
	
	Let $T$ be a once-punctured (topological) torus.
	We call a \emph{curve} a free homotopy class of an unoriented closed curve on $T$,
	and denote the set of curves by $\mathcal{C}(T)$.
	The fundamental group $\pi_1(T)$ is a free group of rank $2$.
	Let $\{ \mathtt{a}, \mathtt{b} \}$ be a generating set of $\pi_1(T)$,
	and let $\mathtt{A} \coloneqq \mathtt{a}^{-1}$ and $\mathtt{B} \coloneqq \mathtt{b}^{-1}$.
	Every element in $\pi_1(T)$ can be written as a word in $\{ \mathtt{a}, \mathtt{b}, \mathtt{A}, \mathtt{B} \}$.
	Conjugacy classes up to inverses of $\pi_1(T)$ are in bijection with $\mathcal{C}(T)$. We call a curve \emph{essential} if it is not null-homotopic or a power of a loop around the puncture. For a word $w$ in $\{\mathtt{a},\mathtt{b},\mathtt{A},\mathtt{B}\}$, we denote by $[w]$ the corresponding element in the free group $\pi_1(T)$ and by $\llbracket w\rrbracket$ the corresponding curve in $\mathcal{C}(T)$.  This equips the set of curves $\mathcal{C}(T)$ with a combinatorial length, the \emph{word length}---i.e.\ the length of a curve is the number of letters in a cyclically reduced representative. We will denote the word length of a curve $\gamma\in\mathcal{C}(T)$ by $\ell(\gamma)$.\\

	The mapping class group $\Mod(T)$ of $T$ is the group of isotopy classes of orientation-preserving self-homeomorphisms of $T$.	The mapping class group acts on $\mathcal{C}(T)$.
	The \emph{topological type}, or simply \emph{type}, of a curve is its orbit under this action.
	Our main result is the following.

	\begin{theorem}\label{thm:mainthm}
		Let $\gamma \in \mathcal{C}(T)$ be an essential curve.
		There exist $N \in \ZZ_{\geq 1}$, $(P_i, Q_i, K_i)_{i=1}^N \in (\ZZ_{\geq 1}^2 \times \ZZ_{\geq 0})^N$, $m\in\{2,4\}$, and $L_0 \in \ZZ_{\geq 1}$ such that, for all integers $L \geq L_0$, we have
		\[
		\# \{ \alpha \in \Mod(T) \cdot \gamma \mid \ell(\alpha) = L \}
		=
		m\left(\sum_{i=1}^{N} \varphi_{P_i, Q_i}(L - 4 K_i)
		+
		C(L) 
		\right),
		\]
		where
		\[
		\varphi_{P, Q}(L)
		\coloneqq
		\# \{ (x, y) \in \ZZ_{\geq1}^2 \mid Px + Qy = L, \ \gcd(x, y) = 1 \},
		\]
		and $C(L) \in \ZZ$ is an explicit periodic function with period at most $\prod_{i=1}^N\mathrm{lcm}(P_i,Q_i)$ satisfying $|C(L)|\leq 2N$. Moreover, 
			$m=2$ if $\mathrm{Stab}_{\mathrm{Mod}(T)}(\gamma)$ has some non-trivial torsion elements, and $m=4$ otherwise. 
	\end{theorem}
	Note that $\varphi_{1,1}$ is nothing but Euler's totient function $\varphi$,
    and more generally, $\varphi_{k,k}(n) = \varphi(n/k)$ for any $k, n \in \mathbb{Z}_{\geq 1}$, where we adopt the convention $\varphi(x) = 0$ for $x \in \mathbb{R}\setminus\ZZ_{>0}$. Moreover, the constant $L_0$ of Theorem \ref{thm:mainthm} can be trivially bounded by an exponential of the minimal length in the orbit (see Lemma \ref{lem:boundL0}). 
	
	For fixed $P,Q$, the function $\varphi_{P,Q}(L)$ counts the coprime positive integer solutions of the linear Diophantine equation $Px+Qy=L$. Geometrically, these are the primitive points---that is, the visible points in Figure \ref{fig:eulerpq}---on the segment that the line $Px+Qy=L$ cuts in the open first quadrant. As $L$ grows, this count varies quasilinearly.
	
	\begin{figure}[ht]\centering
		\begin{overpic}[width=.6\linewidth]{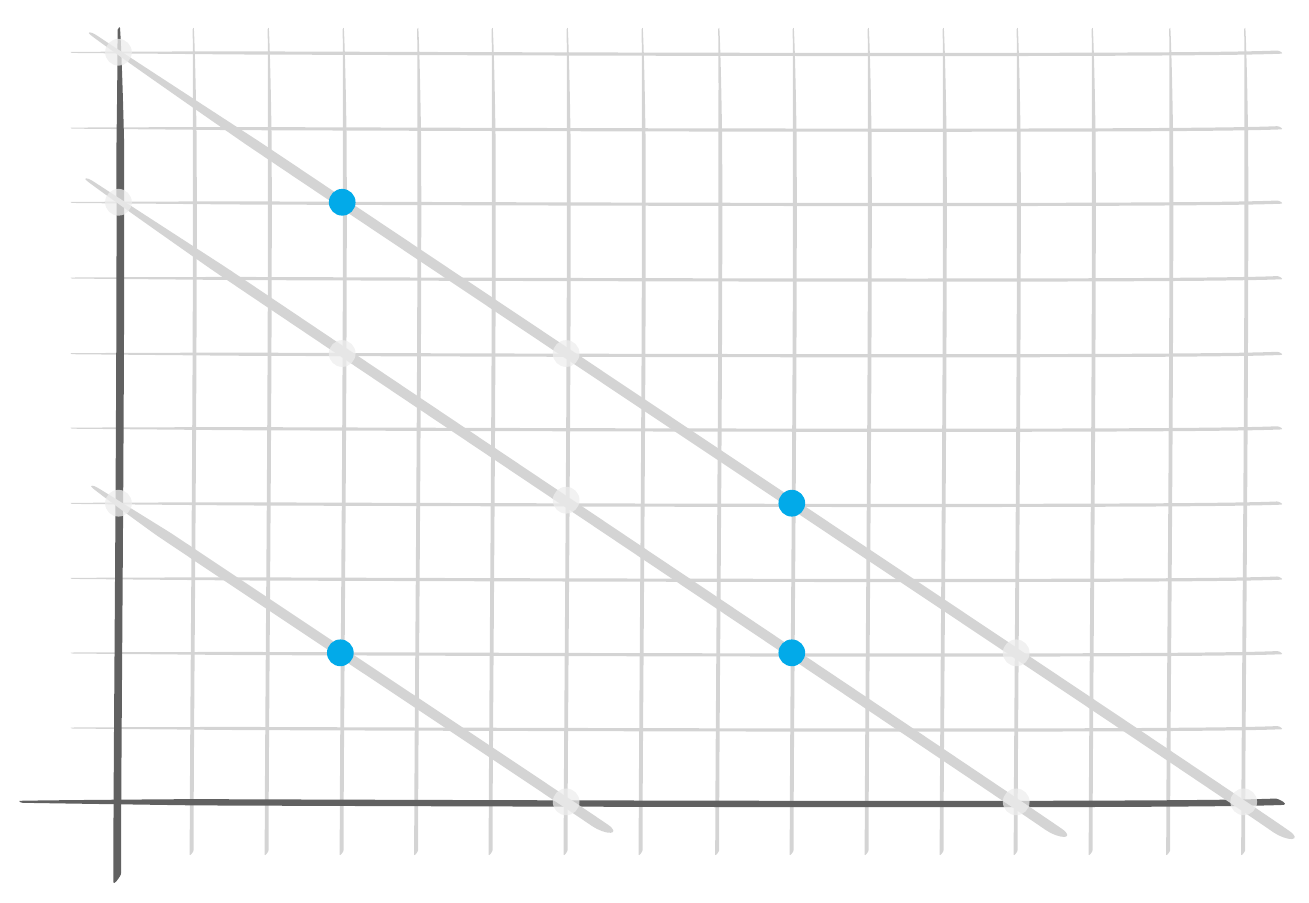}
		\end{overpic}
		\caption{Representation of $\varphi_{2,3}(12)=1$, $\varphi_{2,3}(24)=1$, and $\varphi_{2,3}(30)=2$.}
		\label{fig:eulerpq}
	\end{figure}
	
	\begin{examples}
		The counting functions of types up to self-intersection $3$ were computed by Chas in \cite{Chas}, some were proved and others verified up to long lengths.\\
		
		The algorithm in Section \ref{sec:algorithm} computes the counting formula from Theorem \ref{thm:mainthm} for any type in exponential time with respect to the minimal word length in the orbit, taking less than a second for writeable examples. All of the examples computed by Chas can now be verified. Some of them are the following---all coinciding with Chas' predictions. Denote 
		\[
		N_\gamma(L):=\# \{ \alpha \in \Mod(T) \cdot \gamma \mid \ell(\alpha) = L \}.
		\]
		
		Then,
		
		 \begin{table}[ht]
        \centering
        \begin{tabular}[t]{c|c|c}
            \toprule
            $\gamma$ & $N_L(\gamma)$ & $L_0$ \\
            \midrule
            $\mathtt{a}$ & $2\varphi_{1,1}(L) = 2\varphi(L)$ & $2$ \\
            \midrule
            $\mathtt{aabAB}$ & $4\varphi_{1,1}(L-4) = 4\varphi(L-4)$ & $6$ \\
            \midrule
            $\mathtt{abaBabAB}$ & $4\varphi_{2,2}(L-4) = \varphi((L-4)/2)$ & $9$ \\
            \midrule
            $\mathtt{aaaabb}$ & $4\varphi_{2,2}(L) + 4\varphi_{2,2}(L-4) = 4\varphi(L/2) + 4\varphi((L-4)/2)$ & $7$ \\
            \midrule
            $\mathtt{aabAbaBAb}$ & $4\varphi_{1,1}(L-4) + 4\varphi_{1,1}(L-8) = 4\varphi(L-4) + 4\varphi(L-8)$ & $9$ \\
            \bottomrule
        \end{tabular}
        \caption{Low-intersection examples.}
        \label{tab:exampleslow}
    \end{table}
		Note that all of the counting functions for low-intersection examples are linear combinations of Euler's totient functions. However, this neat form needs to be relaxed as higher-intersection cases appear. \\

		Denote by $[a_0,\hdots, a_{k-1}](L\ \mathrm{mod}\ k)$ the periodic function on $L$ that takes the value $a_i$ when $L\equiv i\pmod n$. That is, 
        \[
            [a_0, \dots, a_{k-1}](L \bmod k)
            \coloneqq
            \begin{cases}
                a_0 & \text{if } L \equiv 0 \bmod k, \\
                a_1 & \text{if } L \equiv 1 \bmod k, \\
                    & \vdots \\
                a_{k-1} & \text{if } L \equiv k-1 \bmod k.
        \end{cases}
        \]
        We can then compute compute any desired examples. In Table \ref{tab:exampleshigh} there are some examples where: either the counting functions are not Euler totient functions but weighted generalizations of it; or periodic corrections appear. Theorem \ref{thm:mainthm} ensures that the counting function of any type can be written in this form.
		
        \begin{table}[ht]
        \centering
        \begin{tabular}[t]{c|c|c}
            \toprule
            $\gamma$ & $N_L(\gamma)$ & $L_0$ \\
            \midrule
            $\mathtt{aaabbabb}$ & $8\varphi_{4,4}(L) + 2\varphi_{2,2}(L-4) + 4 [1, 0, -1, 0](L \bmod 4)$ & $9$ \\
            \midrule
            $\mathtt{aabbAAb}$ & $8\varphi_{2, 3} (L-4) + 8\varphi_{3, 3}(L) + 4 [-2, 0, 2](L \bmod 3)$ & $8$ \\
            \midrule
            $\mathtt{aaabbb}$ & $6\varphi_{3, 3} (L) + 4 \varphi_{2, 3}(L-4)+ 2 [-2, 0, 2](L \bmod 3)$ & $8$ \\
            \bottomrule
        \end{tabular}
        \caption{Higher-intersection examples.}
		\label{tab:exampleshigh}
    \end{table}

	\end{examples}
	
	Moreover, from Theorem \ref{thm:mainthm} we can obtain the asymptotic growth of the counting in terms of the Diophantine coefficients.
	
	\begin{restatable}{corollary}{asymptotic}\label{thm:asym}
		Using the notation of Theorem~\ref{thm:mainthm},
		asymptotically
		\[
		\# \{ \alpha \in \Mod(T) \cdot \gamma \mid \ell(\alpha) \leq L \}
		\sim
		\frac{3m}{\pi^2}
		\left( \sum_{i=1}^{N} \frac{1}{P_i Q_i} \right)
		L^2
		\]
		as $L \to \infty$.
	\end{restatable}
	
	The mapping class group $\mathrm{Mod}(T)$ is isomorphic to $SL_2(\ZZ)$ and is generated by the two Dehn twists around the generators $\mathtt{a}$ and $\mathtt{b}$. For convenience, we will denote these by $L$ and $R$, respectively. On words, we will use the convention that their actions are defined by 
	\[
	L:\begin{matrix}
		& \mathtt{a}&\mapsto &\mathtt{a} \\
		& \mathtt{b}& \mapsto & \mathtt{ab}
	\end{matrix}\quad 
	\text{ and }\quad R:
	\begin{matrix}
		& \mathtt{a}&\mapsto &\mathtt{ab} \\
		& \mathtt{b}& \mapsto & \mathtt{b}.
	\end{matrix}
	\]
	
	The isomorphism $\mathrm{Mod}(T)\leftrightarrow SL_2(\ZZ)$ is then given by $L=\begin{psmallmatrix}
		1 & 1 \\ 
		0 & 1
	\end{psmallmatrix}$ and 
	$R=\begin{psmallmatrix}
		1 & 0 \\ 
		1 & 1
	\end{psmallmatrix}$. It is a classical result (see e.g. \cite[Proposition~3.3]{bram18}) that the semigroup $\langle L, R\rangle_+$ of the positive products of $L$ and $R$ is a free semigroup isomorphic to $SL_2(\NN)$, the semigroup formed by the subset $SL_2(\ZZ)$ with positive matrices. Denote also $S:= \begin{psmallmatrix}
		0 & -1 \\
		1 & 0
	\end{psmallmatrix}$. Note that $S$ is the order-four elliptic element of the mapping class group, corresponding to the renamings of the generators inside their set. For an essential curve $\gamma\in \mathcal{C}(T)$ we will consider two suborbits of $\mathrm{Mod}(T)\cdot \gamma$: $T_\gamma:= \langle L, R\rangle_+ \cdot \gamma$ and $T_{S(\gamma)}:= \langle L, R\rangle_+ \cdot S(\gamma)$. We will see these suborbits as binary infinite rooted trees with vertices labelled by curves and edges corresponding to the $L$,$R$-actions and, unless otherwise stated, not identifying vertices labelled by the same curve. A foundational result for the construction of Theorem \ref{thm:mainthm} is the following. We say that a curve $\gamma\in\mathcal{C}(T)$ is in standard position if it is normalized such that $\mathrm{Stab}_{\mathrm{Mod}(T)}(\gamma)$ is one of the following:
    \begin{equation}\label{eq:stand}\{I\}, \langle -I \rangle, \langle U\rangle, \langle S \rangle, \langle -U\rangle, \langle L\rangle, \langle -L\rangle, \langle -I, L\rangle, \end{equation} where $U:=\begin{psmallmatrix}
		0 & -1 \\ 1 & -1
	\end{psmallmatrix}$.

	\begin{restatable}{proposition}{twotrees}\label{thm:twotrees} For an essential curve $\gamma\in \mathcal{C}(T)$, 
		\begin{enumerate}
			\item Every $\alpha\in \mathrm{Mod}(T)\cdot \gamma$ satisfies $S^k\alpha\in T_\gamma\cup T_{S(\gamma)}$ for some $k\in\{0,1,2,3\}$;
			\item Assume that $\gamma$ is in standard position and has minimal length in the orbit. In the notation of Theorem \ref{thm:mainthm}, for all $L > \ell(\gamma)$,
			\[
			\#\{\alpha\in\mathrm{Mod}(T)\cdot \gamma\mid \ell(\alpha)=L\}=m\cdot \# \{ \alpha \in T_\gamma\cup T_{S(\gamma)} \mid \ell(\alpha) = L \},
			\]
			where 
			$m=2$ if $\mathrm{Stab}_{\mathrm{Mod}(T)}(\gamma)$ has some non-trivial torsion elements, and $m=4$ otherwise. 
		\end{enumerate}
		
	\end{restatable}
	
	Note that $m$ is the same coefficient as in Theorem \ref{thm:mainthm}. The proof of Theorem \ref{thm:mainthm} is based on a notion of stability that controls the growth of lengths at the trees $T_\gamma$ and $T_{S(\gamma)}$ from Proposition \ref{thm:twotrees}. We will denote the concatenation of words $u$, $v$ as a product $uv$. A word in $\{\mathtt{a,b,A,B}\}$ is \emph{monotone} if all exponents are either positive or negative. From now on, we will denote by $\mathtt{z}:=\mathtt{aBAb}$ the commutator and $\mathtt{Z}:=\mathtt{z}^{-1}$. We call a curve $\gamma\in\mathcal{C}(T)$ \emph{stable} if one of its representatives in $\pi_1(T)$ can be written as $\prod_{i=1}^n \mathtt{z}^{e_i}w_i$, with all $e_i\neq0$ when $n>1$, and all $w_i$ non-empty monotone words. 
	
	\begin{restatable}{proposition}{finitestable}\label{prop:nonstable}
		For any essential curve $\gamma\in\mathcal{C}(T)$, there are finitely many elements in $T_\gamma$ that are not stable.
	\end{restatable}
	
	This notion of stability is spread along positive suborbits as follows. Both $[\mathtt{z}]$ and $[\mathtt{Z}]$ are invariant in $\pi_1(T)$ by the $L$,$R$-actions. Since monotonicity is also carried by the $L$,$R$-actions, stability is carried by them too. We call a subtree of $T_\gamma$ \emph{stable} if all of its vertices are stable, and \emph{forward-complete} if it is complete by the positive $L$- and $R$-actions.
	
	\begin{corollary}\label{cor:maxdisjoint}
		The tree $T_\gamma$ is decomposed into finitely-many disjoint, forward-complete, maximal stable subtrees.
	\end{corollary}
	
	Corollary \ref{cor:maxdisjoint} is depicted in Figure \ref{fig:stabletrees}. Take a curve $\gamma$ and set $\tilde\gamma:=S(\gamma)$. Construct the two trees $T_\gamma$ and $T_{\tilde\gamma}$ by applying the $L$,$R$-actions. Proposition \ref{prop:nonstable} says that there are finitely-many unstable curves---corresponding to the vertices outside the blue zone in the figure---and Corollary \ref{cor:maxdisjoint} decomposes these trees, up to the finite unstable complement, into maximal stable trees---like the blue $T_{R(\gamma)}$ tree in the figure.
	
	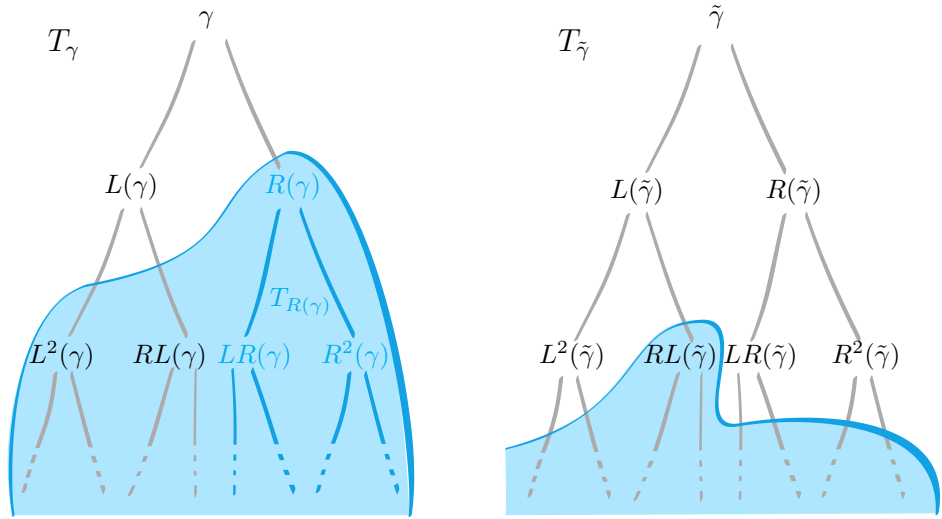
\begin{figure}[h!]\centering
		\begin{overpic}[width=.9\linewidth]{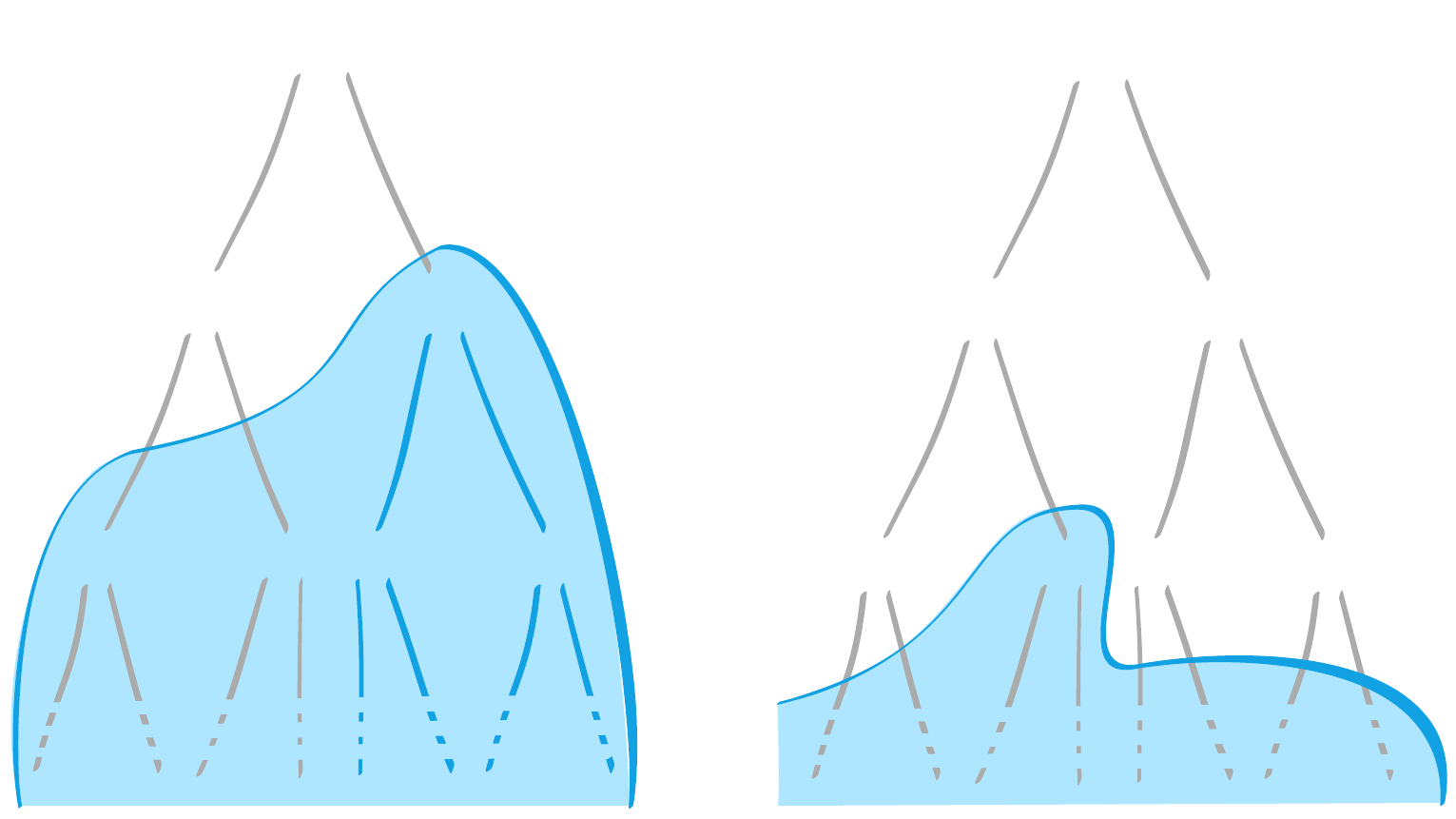}
			\put(5,50){\large$T_\gamma$}
			\put(59,50){\large$T_{\tilde\gamma}$}
			\put(21,53){$\gamma$}
			\put(75,53){$\tilde\gamma$}
			\put(11,35){$L(\gamma)$}
			\put(28,35){\color{cyan}$R(\gamma)$}
			\put(28.5,23){\color{cyan}$T_{R(\gamma)}$}
			\put(64.5,34.5){$L(\tilde\gamma)$}
			\put(81,34.5){$R(\tilde\gamma)$}
			\put(3,17){$L^2(\gamma)$}
			\put(14,17){$RL(\gamma)$}
			\put(23,17){\color{cyan}$LR(\gamma)$}
			\put(34,17){\color{cyan}$R^2(\gamma)$}
			\put(57,17){$L^2(\tilde\gamma)$}
			\put(68,17){$RL(\tilde\gamma)$}
			\put(76.5,17){$LR(\tilde\gamma)$}
			\put(88,17){$R^2(\tilde\gamma)$}
		\end{overpic}
		\caption{Stability of the trees $T_\gamma$ and $T_{\tilde\gamma}$.}
		\label{fig:stabletrees}
	\end{figure}
	
	This will allow us to prove a stronger version of Theorem \ref{thm:mainthm}.
	
	\begin{restatable}{theorem}{theoremsubtree}\label{thm:mainthm_subtree} Let $\gamma\in\mathcal{C}(T)$ be an essential curve. Let $\Gamma \subseteq T_\gamma$ be a maximal stable subtree. There exist $(P,Q,K)\in\ZZ^2_{\geq1}\times \ZZ_{\geq0}$, $L_0\in\ZZ_{\geq1}$ such that, for all integer $L\geq L_0$, we have
		\[
		\# \{ \alpha \in \Gamma \mid \ell(\alpha) = L \}
		=
		\varphi_{P, Q}(L - 4 K)
		+
		C(L),
		\]
		where
		\[
		\varphi_{P, Q}(L)
		\coloneqq
		\# \{ (x, y) \in \ZZ_{>0}^2 \mid Px + Qy = L, \ \gcd(x, y) = 1 \},
		\]
		and $C(L) \in \{-2,\hdots,2\}$ is an explicit periodic function with period at most $\mathrm{lcm}(P,Q)$.
	\end{restatable}
	
	Proposition \ref{thm:twotrees} and  Theorem \ref{thm:mainthm_subtree} show that the counting strategy is rooted to the decomposition of the positive suborbit trees $T_\gamma$ of an essential curve $\gamma\in\mathcal{C}(T)$ into stable regions where length growth is linear, with finite complement bounded in size. We will finally present a topological interpretation of these maximal stable subtrees.\\
	
	We call a \emph{multicurve} a multiset of curves, and denote the set of all multicurves on $T$ by $\mathcal{M}(T)$. We refer as \emph{intersection number} to the geometric intersection number. We denote by $i(\cdot,\cdot)$ the intersection number between two curves and by $i(\cdot)$ the self-intersection number of a curve, it being half of the intersection number with itself. An \emph{intersection point} is a point on $T$ where two (resp. one) representative loops in minimal position attain the intersection (resp. self-intersection). The self-intersections of a multicurve are the union of the pairwise intersections between its components and the self-intersections of each of them. A simple multicurve is a multicurve with zero self-intersections. Given a multicurve, we can define at a representative's self-intersection point two resolutions of the intersection as local surgeries that decrease the intersection number (see e.g. \cite{Turaev1991}). This resolution can be extended from a representative to a free homotopy class, and from a free homotopy class to the entire $\mathrm{Mod}(T)$-orbit by equivariancy. See Section \ref{sec:resolutions} for a formal explanation of this. 

    \begin{figure}[h!]\centering
		\begin{overpic}[width=.75\linewidth]{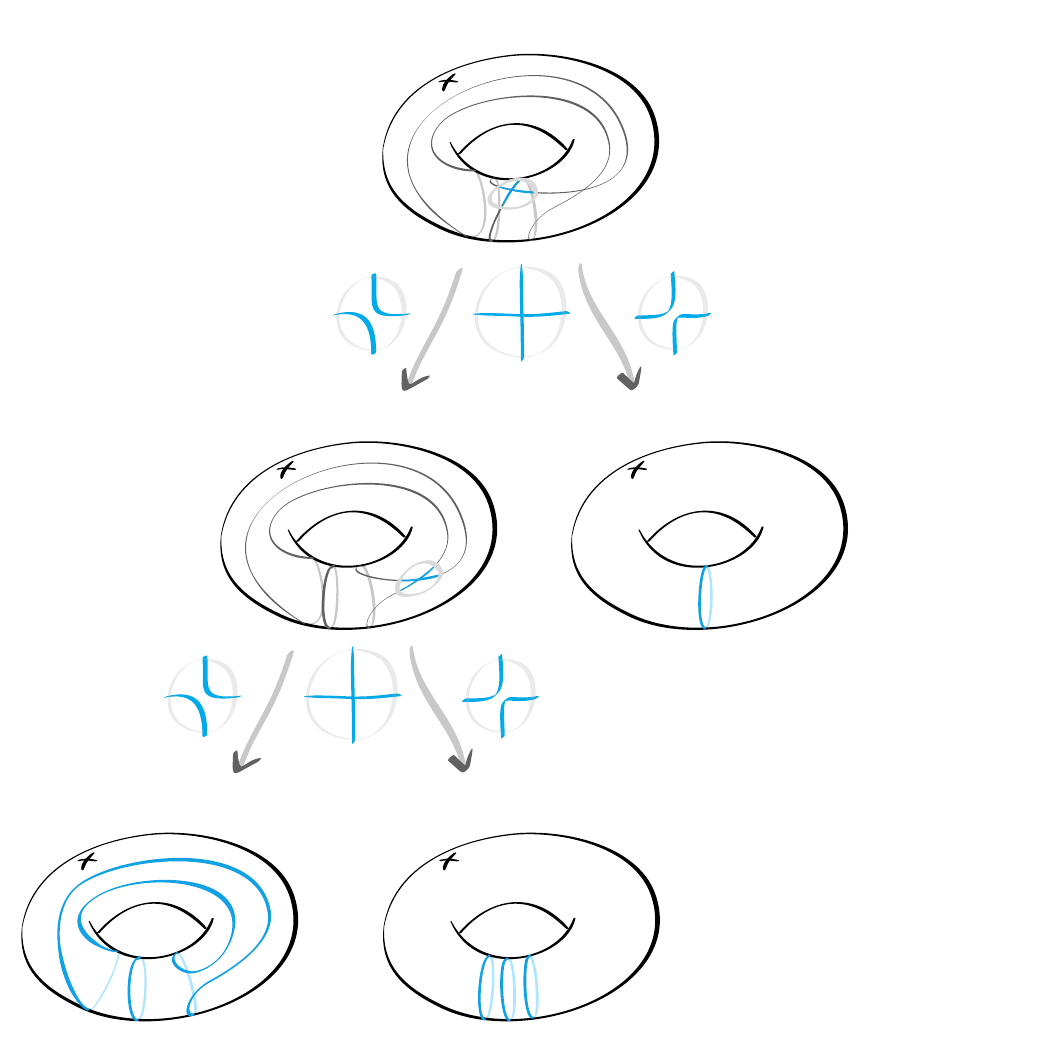}
			\put(44,96){$\llbracket \mathtt{aabaB}\rrbracket$}
			\put(25.5,58.5){$\{\llbracket \mathtt{a}\rrbracket,\llbracket \mathtt{abaB}\rrbracket\}$}
			\put(64,58.5){$\{\llbracket \mathtt{a}\rrbracket\}$}
			\put(7,21){$\{\llbracket \mathtt{a}\rrbracket,\llbracket \mathtt{abAB}\rrbracket\}$}
			\put(40,21){$\{\llbracket \mathtt{a}\rrbracket,\llbracket \mathtt{a}\rrbracket,\llbracket \mathtt{a}\rrbracket\}$}
		\end{overpic}
		\caption{A list of full resolutions of $\llbracket \mathtt{aabaB}\rrbracket$.}
		\label{fig:resolutionlist}
	\end{figure}
	
	The following resolution procedure is exemplified in Figure \ref{fig:resolutionlist}. Starting at an essential curve $\gamma\in\mathcal{C}(T)$, we choose a self-intersection point and resolve it in the two ways, giving rise to two multicurves. We do this iteratively until we get at most $2^{i(\gamma)}$ resolutions to simple multicurves. We call a \emph{complete list of full resolutions} of $\gamma$ a finite set $\mathcal{R}$ obtained in this way in which, for any $r\in \mathcal{R}$, $r:\mathrm{Mod}(T)\cdot \gamma \to \mathrm{Mod}(T)\cdot r(\gamma)\subseteq \mathcal{M}(T)$ is defined equivariantly---i.e. $r(f(\gamma))=f(r(\gamma))$ for any $f\in\mathrm{Mod}(T)$. Note that such a list might depend on the choices of which intersections to resolve at each step.

    In Figure \ref{fig:resolutionlist}, the full resolution list $\mathcal{R}=\{r_1,r_2,r_3:\mathrm{Mod}(T)\cdot \gamma\to \mathcal{M}(T)\}$ found for $\gamma=\llbracket \mathtt{aabaB}\rrbracket$ is the following: for all $f\in\mathrm{Mod}(T)$, $r_1(f(\gamma))=f(\{\llbracket \mathtt{a}\rrbracket,\llbracket \mathtt{abAB}\rrbracket\})$, $r_2(f(\gamma))=f(\{\llbracket \mathtt{a}\rrbracket,\llbracket \mathtt{a}\rrbracket,\llbracket \mathtt{a}\rrbracket\})$, and $r_3(f(\gamma))=f(\{\llbracket \mathtt{a}\rrbracket\})$.

	\begin{restatable}{theorem}{resolutions}\label{thm:resolutions}
		Let $\gamma\in\mathcal{C}(T)$ be an essential curve. Let $\mathcal{R}$ be a list of full resolutions of $\gamma$. Each maximal stable subtree of $T_\gamma$ is a maximal connected, forward-complete subtree such that a resolution $r\in\mathcal{R}$ is length-preserving at its interior. 
	\end{restatable}
	
	Theorem \ref{thm:resolutions} allows us to describe the coefficients of the counting formula in terms of the predominant resolution's image at the root of the stable subtrees---see Corollary \ref{cor:coeff}.\\
	
	The last section of this work, Section \ref{sec:algorithm}, will present an algorithm for words that generates the formula of Theorem \ref{thm:mainthm} in exponential computational time with respect to the word length. \\

	\subsection*{Related work}
	
	Counting curves on a surface of a given type and bounded length was first studied by McShane and Rivin in \cite{McSR95} for simple curves on hyperbolic once-punctured tori and by Mirzakhani for general hyperbolic orientable finite-type surfaces in \cite{Mirzakhani:first,Mirzakhani:last}. Denote by $\Sigma_{g,n}$ a topological surface of genus $g$ and $n$ punctures. With many generalizations, the one we are interested in is Erlandsson, Parlier and Souto's result \cite{ErlandssonParlierSouto20} (see also a more extensive work in \cite{ErlandssonSouto}). They generalize Mirzakhani's result to any positive, continuous and homogeneous function $\tilde{\ell}$ on the space of geodesic currents of the surface, as a generalization of hyperbolic lengths on the curves on the surface. They prove that for a given mapping class group orbit of curves $\mathrm{Mod}(\Sigma_{g,n})\cdot\gamma$ there exist constants $C_{g,n;\gamma}$ depending only on $g,n$ and $\gamma\in\mathcal{C}(\Sigma_{g,n})$, and $B_{\tilde\ell}$ depending only on the function $\tilde\ell$, such that
	\[
	| \{ \alpha \in \mathcal{C}(\Sigma_{g,n}) \mid \alpha \in\mathrm{Mod}(\Sigma_{g,n})\cdot \gamma , \, \tilde\ell(\alpha) \leq L \}|
	\sim
	C_{g,n;\gamma} \cdot B_{\tilde\ell} \cdot L^{6g-6+2n},
	\]
	
	where $\mathcal{C}(\Sigma_{g,n})$ is the set of closed curves up to free homotopy of the surface. In particular, this theorem applies to the word length $\ell$, i.e. the function that measures the number of letters of a fixed generating set of the fundamental group needed to represent a curve (see e.g. \cite{Erlandsson}). This length allows to study the combinatorial structure of curves on a surface and compare their length and intersection, all while being quasi-isometric to hyperbolic length for compact surfaces. 
	
	Recall that we denote by $T$ the once-punctured torus and by $\ell$ the word length with respect to a fixed set of generators of the fundamental group. By \cite{ErlandssonParlierSouto20}, we know that for every curve type $\mathrm{Mod}(T)\cdot \gamma$, with $\gamma\in\mathcal{C}(T)$, there exists a positive constant $C_\gamma$ such that 
	\[
	| \{ \alpha \in \mathcal{C}(T) \mid \alpha \in\mathrm{Mod}(T)\cdot \gamma , \, \ell(\alpha) \leq L \}|
	\sim
	C_{\gamma}  \cdot L^{2}.
	\]
	In Corollary \ref{thm:asym}, we find $C_\gamma$ in terms of the finite number of coefficients found by the algorithm in Section \ref{sec:algorithm}.
	The word length gives one hope that is not achievable with hyperbolic lengths: finding a closed expression for the counting instead of the asymptotic growth. For the low-intersection types, the combinatorics of the words generating them are being studied. See for example \cite{BuserSemmler,FisacLiu}, for a combinatorial description of all words with self-intersection $0$ and $1$, respectively. See also \cite{CP10,Cha15,CML12,CL12,Chas} for examples of how to compare word length and intersection on the once-punctured torus or a pair of pants and computational results about the number of curves of a given type and length, and see also \cite{FisacLiu} for how to derive the counting of curves from the combinatorial information in the intersection $0$ and $1$ cases.
	
	Before this work, the existence, shape and computation of such an explicit counting function was only known for specific low-intersection cases. The main result, Theorem \ref{thm:mainthm} settles a conjecture posed by Chas in \cite{Chas}, generalizing what empirical evidence hints for low-intersection types.
	
	\subsection*{Statement on AI}
	Throughout this project LLMs have been used extensively as a formula-guessing software and as a coder for computations and verifications. All the writing, mathematical arguments and proofs have been made independently.
	
	\subsection*{Acknowledgements}
	FB is supported by the Simons Foundation. DF is funded by the ANR grant GALS (ANR-23-CE40-0001). The authors would like to thank Juan Souto for interesting yet confusing conversations.
	
	\section{Word stability}
	
	Recall that a curve is stable if one of its representatives in $\pi_1(T)$ can be written as $\prod_{i=1}^n \mathtt{z}^{e_i}w_i$, with $\mathtt{z}=[\mathtt{a},\mathtt{B}]$, all $e_i\neq0$ and all $w_i$ non-empty monotone words in $\{\mathtt{a},\mathtt{b},\mathtt{A},\mathtt{B}\}$. We call the $\pi_1(T)$ element a \emph{stable element}, and its representative in stable form a \emph{stable word}. Moreover, we will denote by $|w|_S$ the number of occurrences of the letters $S\subseteq \{\mathtt{a},\mathtt{b},\mathtt{A},\mathtt{B}\}$ at the word $w$, and, for a simple curve $\alpha\in\mathcal{C}(T)$, by $|\alpha|_\mathtt{x}$ the number of occurrences of the letters $\{\mathtt{x},\mathtt{X}\}$ in a cyclically reduced representative.
	
	\begin{lemma}\label{lem:stable}
		If $w$ is a stable word, then $[L(w)],[R(w)]$ are stable elements.
	\end{lemma}
	
	\begin{proof}
		Note that $[\mathtt{z}]=L([\mathtt{z}])=R([\mathtt{z}])$ and $[\mathtt{Z}]=L([\mathtt{Z}])=R([\mathtt{Z}])$. Moreover, if $w$ is monotone, then so are $L(w)$ and $R(w)$.
	\end{proof}
	
	This allows us to prove the next proposition. We will use train-tracks and pre-train-tracks during this article. These two objects are usually defined on a smooth structure (see e.g. \cite{ErlandssonSouto}). We will now give an equivalent purely topological definition. We will call \emph{pre-train-track} a connected embedded finite graph $\tau$ in $T$ equipped with a labelling on the set of half-edges satisfying some conditions. We call its vertices $V(\tau)$ \emph{switches}, and its edges $E(\tau)$ \emph{branches}. Denote for a switch $v\in V(\tau)$ the set of incident half-edges by $H(v)$. Consider a partition of each set $H(v)$ into $\mathrm{In}(v)\sqcup \mathrm{Out}(v)$. This embedded graph $\tau$ equipped with the partition at every set $\{H(v)=\mathrm{In}(v)\sqcup \mathrm{Out(v)},v\in V(\tau)\}$ is called a \emph{pre-train-track} if it satisfies: 
	\begin{enumerate}
		\item For each $v\in V(\tau)$, $\mathrm{In}(v)$ and $\mathrm{Out}(v)$ are non-empty;
		\item For each $v\in V(\tau)$, the elements of $\mathrm{In}(v)$ (resp. of $\mathrm{Out}(v)$) are consecutive with respect to the cyclic ordering of $H(v)$ induced by the embedding of $\tau$.
	\end{enumerate}
	
	Moreover, consider now the closures of the connected components of $T\setminus \tau$, where if a vertex is attained between two incoming or two outgoing half-edges, it is called a cusp of the complementary region. Then, $\tau$ is called a \emph{train-track} if none of its complementary regions are disks with at most 2 cusps, or punctured disks with no cusp. 
	
	On a path on a pre-train-track $\tau$, a \emph{legal} turn is a passing through a vertex from an incoming half-edge to an outgoing half-edge, or the other way around. A turn is \emph{illegal} otherwise. As usual, we say that a curve $\gamma\in\mathcal{C}(T)$ is \emph{carried} by a pre-train-track $\tau$ if it has a representative loop contained all in $\tau$ with no illegal turns. Note that the train-track condition ensures that there is at most a unique way to push a curve onto $\tau$ with no backtracking.

    \finitestable*
	
	\begin{proof}
		During this proof we will use a pre-train-track as visual aid. Note that all the arguments can be written in terms of combinatorics of words. A curve is stable if and only if it is carried by the pre-train-track $\tau$ from Figure \ref{fig:pretrain}.
		
		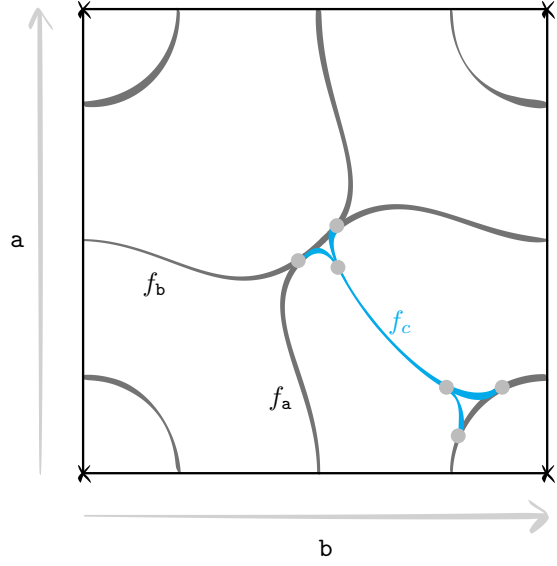
\begin{figure}[h!]\centering
			\begin{overpic}[width=.6\linewidth]{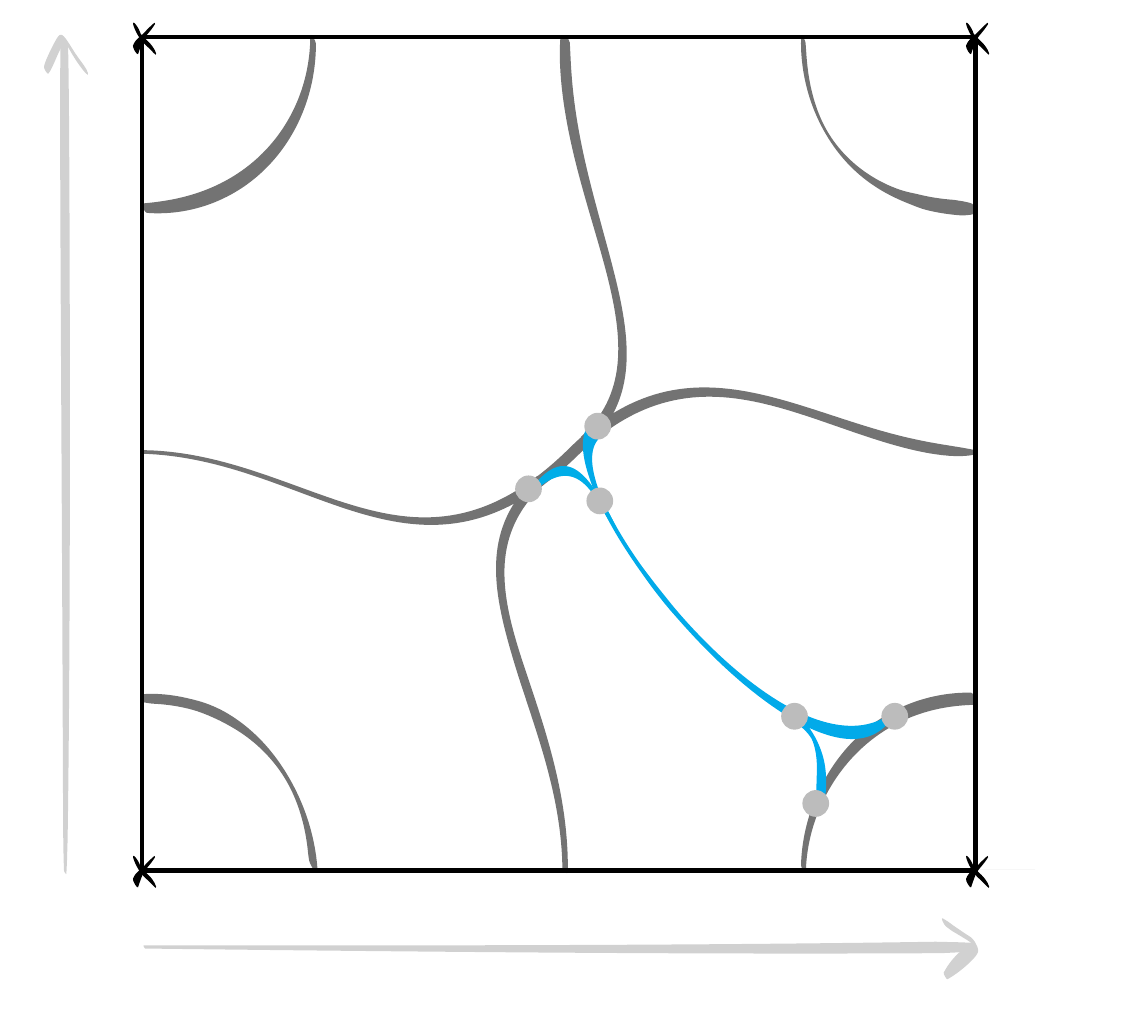}
				\put(1,50){$\mathtt{a}$}	
				\put(50,1){$\mathtt{b}$}
				\put(42,25){$f_\mathtt{a}$}
				\put(22,43){$f_\mathtt{b}$}
				\put(61,37){\color{cyan}$f_c$}
			\end{overpic}
			\caption{The pre-train-track defining admissibility.}
			\label{fig:pretrain}
		\end{figure}
		
		Define an \emph{admissible push} of a curve $\gamma$ as a free homotopy that sends $\gamma$ onto a loop on $\tau$ with no backtracking such that, fixing any orientation of $\gamma$, when considering the sub-paths on $\tau\setminus \{f_c\}$, every connected component passes through $f_\mathtt{a}$ (resp. $f_\mathtt{b}$) with one orientation, and such that all illegal turns are made between $f_\mathtt{a}$ and $f_\mathtt{b}$.

        Hence, an admissible push of a curve $\gamma$ corresponds to writing it as a word $\prod_{i=1}^n \mathtt{z}^{e_i}w_i$ with $e_i\neq0$ and such that each $w_i$ is non-empty and has all $\mathtt{a}$'s powers of the same sign, and all $\mathtt{b}$'s powers of the same sign. We call such a word to be in \emph{admissible form} or just an \emph{admissible word}. A subword $w_i$ of an admissible word is called \emph{monotone} if the signs of the $\mathtt{a}$'s and of the $\mathtt{b}$'s are the same, and \emph{mixed} otherwise.

        The existence of such an admissible push is guaranteed combinatorially by some word transformations acting trivially in $\mathcal{C}(T)$.

        \begin{lemma}\label{lem:admissible}
            Every cyclically reduced cyclic word in $\{\mathtt{a},\mathtt{b},\mathtt{A},\mathtt{B}\}$ can be rewritten in admissible form by finitely many of the ($\mathcal{C}(T)$-trivial) rewriting movements (\ref{eq:leftrewr})-(\ref{eq:rightrewr}) below and giving preference to the (\ref{eq:leftrewr}) ones over the (\ref{eq:rightrewr}) ones. Moreover, no $\mathtt{zZ}$ or $\mathtt{Zz}$ subword will appear in the process.
        \end{lemma}

        The rewritings are: 
        \begin{align}
            \mathtt{aB}^kA &\to (\mathtt{zB})^k, & \mathtt{ab}^kA &\to (\mathtt{bZ})^k, & \mathtt{Ba}^kb &\to (\mathtt{Za})^k, & \mathtt{BA}^kb &\to (\mathtt{Az})^k; \label{eq:leftrewr}\\
            \mathtt{Ab}^k\mathtt{a} &\to \mathtt{bAzb}^{k-1}\mathtt{a}, & \mathtt{AB}^k\mathtt{a} &\to \mathtt{AB}^{k-1}\mathtt{ZaB}, &\mathtt{ba}^k\mathtt{B} &\to \mathtt{ba}^{k-1}\mathtt{zBa}, &\mathtt{bA}^k\mathtt{B} &\to \mathtt{AbZA}^{k-1}\mathtt{B}.&\label{eq:rightrewr}
        \end{align}

        \begin{proof}[Proof of Lemma \ref{lem:admissible}]
        For any cyclically reduced word $w$ in $\{\mathtt{a,b,z,A,B,Z}\}$, define $\eta_\mathtt{a}(w)$ as the number of subwords of the form $\mathtt{ab}^k\mathtt{A}$ or $\mathtt{Ab}^k\mathtt{a}$ for some $k\neq0$. Define $\eta_\mathtt{b}(w)$ symmetrically and \[\eta(w):=\eta_\mathtt{a}(w)+\eta_\mathtt{b}(w).\]
        Note that $\eta(w)=0$ if and only if the word is in admissible form. Also, $\eta>0$ if and only if one of the above rewritings is possible, and in that case the rewritings always decrease $\eta$: the first two of every line decrease $\eta_\mathtt{a}$ and don't increase $\eta_\mathtt{b}$, and the last two of every line behave symmetrically.

        Finally, no $\mathtt{zZ}$ or $\mathtt{Zz}$ will be created during this process because the (\ref{eq:leftrewr}) rewritings do not admit subword combinations with commutator cancellations, and for that to happen with the (\ref{eq:rightrewr}) rewritings the word would first admit a (\ref{eq:leftrewr}) rewriting which has been given preference.
            
        \end{proof}
        
        Recall that we call the representative of an admissible push a word in \emph{admissible form}. A subword is then non-admissible if it has any subword of the form $\mathtt{ab}^k\mathtt{A}$, $\mathtt{Ab}^k\mathtt{a}$, $\mathtt{ba}^k\mathtt{B}$, or $\mathtt{Ba}^k\mathtt{b}$  for any $k\neq0$.
		
		For such an admissible push, define its complexity as 
		\[
		c\left(\prod_{i=1}^n \mathtt{z}^{e_i}w_i\right):=\sum_{w_i\text{ mixed}}\ell(w_i).
		\]
		And the complexity of an element $[w]\in\pi_1(T)$ as 
		\[
		c([w]):=\min\{c(w')\mid w'\text{ admissible push of }w\}.
		\]
		Thus, given an essential curve $\gamma\in\mathcal{C}(T)$, the proposition follows from the following claim.
		
		\textbf{Claim:} If $c([w])>0$, then $c([L(w)]),c([R(w)])<c([w])$.
		This is achieved by re-pushing the non-admissible subwords of the representatives of $[L(w)]$ and $[R(w)]$  without changing the already-admissible part. Assume without loss of generality that $w_i=\mathtt{a}^{n_1}\mathtt{B}^{m_1}\cdots \mathtt{a}^{n_k}\mathtt{B}^{m_k}$, with $n_1,m_k\geq0$ and all the rest of the exponents are strictly positive. Then $R(w_i)=(\mathtt{ab})^{n_1}\mathtt{B}^{m_1}\cdots (\mathtt{ab})^{n_k}\mathtt{B}^{m_k}$ and $L(w_i)=\mathtt{a}^{n_1}(\mathtt{BA})^{m_1}\cdots \mathtt{a}^{n_k}(\mathtt{BA})^{m_k}$. Denote by $\overline{R(w_i)}$ and $\overline{L(w_i)}$ their reduced representatives. For $[R(w_i)]$, applying the rewritings $\mathtt{ba}^k\mathtt{B}\mapsto\mathtt{ba}^{k-1}\mathtt{zBa}$, and $\mathtt{Ba}^k\mathtt{b}\mapsto(\mathtt{Za})^k$ we can make $[R(w_i)]=[\widetilde{R(w_i)}]$ admissible, where $\widetilde{R(w_i)}$ is a representative in admissible form after applying only these rewritings to $\overline{R(w_i)}$. Note that, for $w_i\neq \mathtt{a}^n\mathtt{B}^m$,
		\begin{align*} |\widetilde{R(w_i)}|_\mathtt{a}=|\overline{R(w_i)}|_\mathtt{a}=|w_i|_\mathtt{a}\text{, and}\\
			|\widetilde{R(w_i)}|_\mathtt{B}\leq|\overline{R(w_i)}|_\mathtt{B}<|w_i|_\mathtt{B}.
		\end{align*}
		And, for $w_i= \mathtt{a}^n\mathtt{B}^m$, the first inequality in the second line becomes strict, and the second one is not strict anymore. This already proves that the word length of the new mixed subwords that will be composed by $\mathtt{a}$ and $\mathtt{B}$ will decrease strictly.
		
		For $[L(w_i)]$, apply instead the rewritings $\mathtt{aB}^k\mathtt{A}\mapsto(\mathtt{zB})^k$ and $\mathtt{AB}^k\mathtt{a}\mapsto\mathtt{AB}^{k-1}\mathtt{ZaB}$ and a symmetric argument holds. 
		
		Note that, as in Lemma \ref{lem:admissible}, with these rewritings we will never land in a situation with $\mathtt{zZ}$ or $\mathtt{Zz}$ as a subword falsely separating monotone words.
		
		The analogous visual image of this proof at the pre-train-track consists on starting with an admissible push of a curve $\gamma\in\mathcal{C}(T)$, and re-pushing $L(\gamma)$ and $R(\gamma)$ by leaving the stable parts invariant such that the combinatorial length on the graph of the non-carried subpaths is strictly decreasing, ensuring that it finishes in finite time.
	\end{proof}	
	
	Corollary \ref{cor:maxdisjoint} is a direct consequence of Proposition \ref{prop:nonstable} and Lemma \ref{lem:stable}.	We will now proceed to prove Proposition \ref{thm:twotrees}. For an element $M\in SL_2(\ZZ)$, write $M=(M_1, M_2)$, with $M_i$ being column vectors. Define $Q_1=\{(x,y)\in \ZZ^2\mid x,y\geq0\}$, $Q_2=\{(x,y)\in \ZZ^2\mid x\geq0,y\leq0\}$, $Q_3=\{(x,y)\in \ZZ^2\mid x,y\leq0\}$, $Q_4=\{(x,y)\in \ZZ^2\mid x\leq0,y\geq0\}$, the four closed quadrants of $\ZZ^2$.
	
	\begin{lemma}\label{lem:nonadj}
		For any $M=(M_1,M_2)\in SL_2(\ZZ)$, $M_1,M_2$ are either at the same quadrant or at opposite quadrants. 
	\end{lemma}
	
	\begin{proof}
		Assume otherwise. Since left multiplication by $S^{j}$ permutes quadrants and preserves the determinant, we can assume without loss of generality that $M_1\in Q_1$ and $M_2\in Q_4$, outside of the axes. Hence, $M_1=(x,y)$, $x,y\geq 1$ and $M_2=(-x',y')$, $x',y'\geq1$. However, then 
		\[
		\det M=xy'+yx'\geq2.
		\]
	\end{proof}
	
	And that leads to a clean factorization of the entire $SL_2(\ZZ)$ as follows. 
	
	\begin{lemma}\label{lem:SLfact}With the notation above, we can write
		\[SL_2(\ZZ)= \bigcup_{i=0}^3 S^i\langle L,R\rangle_+\cup \bigcup_{i=0}^3  S^i\langle L,R\rangle_+ S. \]
		Moreover, the expression $M=S^i M' S^\varepsilon$ with $M'\in\langle L,R\rangle_+$, $i\in \{0,1,2,3\}$, and $\varepsilon\in\{0,1\}$ is unique except for the four rotations $\langle S \rangle$ that have exactly two expressions.
	\end{lemma}
	
	\begin{proof}
		For the existence, by Lemma \ref{lem:nonadj}, either $M$ or $M S^{-1}$ has both columns at the same quadrant. Moreover, left multiplication by $S^{i}$ rotates them to $Q_1$ for some $i\in \{0,1,2,3\}$. Undoing these steps leads to the expression.
		
		For the uniqueness, left multiplication by $S^i$  permutes quadrants. Hence the pair of quadrants occupied determines $i$ once $\varepsilon$ is known. Moreover, $S$ sends $(M_1,M_2)$ to $(-M_2,M_1)$, moving the first column to the opposite quadrant. Hence, $\varepsilon=0$ and $\varepsilon=1$ give different quadrant configurations, unless a column is on the axis. In $\langle L, R\rangle_+$, the elements on the axis are exactly the pure powers $L^n$ and $R^n$. From these, only $M=I$ has both columns on axes, and hence the relation $S^i I S=S^{i+1} I$ gives the four elements represented twice. For every other $M$, the triple $(i,M',\varepsilon)$ is determined.
	\end{proof}
	
	We will call the "$i$" coefficient from Lemma \ref{lem:SLfact} of an element $M\in \SL_2(\ZZ)$ its \emph{rotation}, and denote it by $\mathrm{rot}(M)$. Note that this is a well-defined number for all $M\notin \langle S \rangle$, and is a two-element set otherwise. We need to prove one last ingredient before proving Proposition \ref{thm:twotrees}.

    \begin{lemma} Let $\gamma\in\mathcal{C}(T)$ be an essential curve in standard position with stabilizer $|\mathrm{Stab}_{\mathrm{Mod}(T)}(\gamma)|>2$. Then $\gamma$ has minimal length in $\mathrm{Mod}(T)\cdot \gamma$.   
    \end{lemma}
    \begin{proof}
        By Whitehead's \emph{Peak Reduction Lemma}, if a curve $\gamma$ does not have minimal length, then at least one of the four moves $L,L^{-1},R,R^{-1}$ strictly decreases length. Hence, it is enough to prove that $\ell(\gamma)$ is a local minimum. Moreover, by convexity along the $R$ and $L$ geodesics, $\ell(R\gamma)+\ell(R^{-1}\gamma)\geq 2\ell(\gamma)$, and $\ell(L\gamma)+\ell(L^{-1}\gamma)\geq 2\ell(\gamma)$.\\
        For this proof we will use the following relations: 
        \[
        UL=S,\qquad RS=U,\qquad SLS^{-1}=R^{-1},\qquad SRS^{-1}=L^{-1},
        \]
        and the fact that the word length is $S$-invariant. 
        Assume first that $S\gamma=\gamma$. Then, 
        \[
        \ell(R\gamma)=\ell(SR\gamma)=\ell(L^{-1}S\gamma)=\ell(L^{-1}\gamma).
        \]
        Similarly, $\ell(R^{-1}\gamma)=\ell(L\gamma)$. Then, \[\ell(R\gamma)+\ell(R^{-1}\gamma), \ell(L\gamma)+\ell(R^{-1}\gamma), \ell(R\gamma)+\ell(L^{-1}\gamma), \ell(L\gamma)+\ell(L^{-1}\gamma)\geq 2\ell(\gamma)\] and these four left-hand expressions take only two values, hence none of the neighbouring lengths  $\ell(L\gamma),\ell(R\gamma),\ell(L^{-1}\gamma),\ell(R^{-1}\gamma)$ can be smaller than $\ell(\gamma)$. Assume now that $U\gamma=\gamma$. Then, 
        \[
        \ell(R^{-1}\gamma)=\ell(R^{-1}U\gamma)=\ell(S\gamma)=\ell(\gamma).
        \]
        Similarly, $\ell(L^{-1})=\ell(\gamma).$ Hence, $\ell(R\gamma)\geq 2\ell(\gamma)-\ell(R^{-1}\gamma)=\ell(\gamma)$,
        and similarly $\ell(L\gamma)\geq\ell(\gamma)$. Assume finally that $L\gamma=\gamma$. Then, $\gamma$ is disjoint from $\llbracket\mathtt{a}\rrbracket$, and hence can be written in $\langle \mathtt{a},\mathtt{z}\rangle$, a generating set of the fundamental group of the pair of pants obtained by cutting $T$ along $\llbracket\mathtt{a}\rrbracket$. Then, the cancellations created by $R$ and $R^{-1}$ are in length-preserving bijection, and so $\ell(R\gamma)=\ell(R^{-1}\gamma)$. Hence, by $\ell(R\gamma)+\ell(R^{-1}\gamma)\geq 2\ell(\gamma)$, they are both bigger than $\ell(\gamma)$.
    \end{proof}
    
    We are finally ready to prove the structural proposition. We will present it again.
	
	\twotrees*
	
	\begin{proof}Let us prove both steps in order.
		\begin{enumerate}
			\item Let $\alpha=M \gamma$, for some $M\in SL_2(\ZZ)$. By Lemma \ref{lem:SLfact}, write $M=S^i M' S^\varepsilon$ for $i\in\{0,1,2,3\}$, $M'\in \langle L,R\rangle_+$ and $\varepsilon\in\{0,1\}$. Then, 
			\[
			S^{-i} \alpha =M' S^\varepsilon \gamma\in \langle L,R\rangle_+\cdot  \gamma \cup \langle L,R \rangle_+ S \cdot \gamma.
			\]
			\item 
			The main idea of this proof is the fact that extra symmetries where the index of the covering is not constant happen only at the normalized root, and these symmetries are given by some torsion elements of $\mathrm{Mod}(T)$. Because $\langle S \rangle$ acts as a permutation of letters, each $\langle S \rangle$-orbit inside $\mathrm{Mod}(T)\cdot \gamma$ consists of curves of constant length, so the set of elements of length $L$ is a union of $\langle S \rangle$-orbits. Fix such an orbit $\langle S \rangle  \alpha$, with $\ell(\alpha)=L$. From now on, consider $T_\gamma\cup T_{S(\gamma)}$ as a set of curves. By Part 1, $\langle S \rangle  \alpha\cap (T_\gamma\cup T_{S(\gamma)})\neq\emptyset$. Define $s(\alpha):= |\langle S \rangle \alpha \cap (T_\gamma\cup T_{S(\gamma)})|$. Then, 
			\[
			\#\{\alpha\in\mathrm{Mod}(T)\cdot \gamma\mid \ell(\alpha)=L\}=\sum_{\substack{\alpha\in T_\gamma\cup T_{S(\gamma)} \\ \ell(\alpha)=L}} \frac{|\langle S \rangle \alpha| }{s(\alpha)}.
			\]
			
			Set $J(\alpha):= \{j\in\ZZ/4\mid S^j\alpha\in T_\gamma \cup T_{S(\gamma)}\}$ containing $0$. Set also $H(\alpha):= \{j\in \ZZ/4\mid S^j\alpha=\alpha\}$. We can rewrite 
			\[
			\frac{|\langle S \rangle \alpha| }{s(\alpha)} = \frac{4/|H(\alpha)|}{|J(\alpha)|/|H(\alpha)|}=\frac{4}{|J(\alpha)|}.
			\]
			Hence, we want to prove that $4/|J(\alpha)|$ is constant for $\alpha$ without minimal length.
			Recall that $\gamma$ is in standard position: meaning that $\mathrm{Stab}_{\mathrm{Mod}(T)}(\gamma)$ is one of List (\ref{eq:stand}). The rest of the proof will follow three steps: Write $\alpha=M \gamma$, with $M\in\mathrm{Mod}(T)$. Then, the rotation $\mathrm{rot}(M)$ of $M$ depends only on the first column of $M$, $Me_1$, when it is outside of the axes; every contribution to $|J(\alpha)|$ comes from a rotation which itself can be associated to a stabilizer in $\mathrm{Stab}_{\mathrm{Mod}(T)}(\gamma)$; and, in standard position, those contributions depend only on $\gamma$.

            For the first step: Use Lemma \ref{lem:SLfact} to write $M=S^i M' S^\varepsilon$, with $M'\in\langle L, R\rangle_+$, so both $M'e_i$ and $M'S e_1=M'e_2$ are non-negative vectors. Also, $S$ permutes quadrants, hence, if $Me_1$ does not contain a zero, then $\mathrm{rot}(M)$ is determined by $Me_1$. 
            Also, in that case, $\mathrm{rot}(M)\in\{0,2\}$ if $Me_1$ are of equal sign and $\mathrm{rot}(M)\in\{1,3\}$ if they have different signs.

            For the second step: we want to prove that, for $\alpha\in T_\gamma\cup T_{S(\gamma)}\setminus \langle S\rangle \gamma$, 
            \[
            J(\alpha)=\{-\mathrm{rot}(Mh^{-1})\mid h\in\mathrm{Stab}_{\mathrm{Mod}(T)}(\gamma)\}.
            \]

            Write $\alpha=M\gamma$ for $M=M'S^\varepsilon$ with $M'\in\langle L,R\rangle_+$, $\varepsilon\in\{0,1\}$. This corresponds to the trivial fact that $0\in J(\alpha)$, which corresponds to $-\mathrm{rot}(M)=0$ for $I\in\mathrm{Stab}_{\mathrm{Mod}(T)}(\gamma)$. 
            
            Take now $j\in J(\alpha)\setminus\{0\}$. Then we can also write $S^j\alpha=N\gamma\in T_\gamma\cup T_{S(\gamma)}$. Hence, $N\in\langle L,R\rangle_+$ or $N=N'S$ with $N'\in\langle L,R\rangle_+$. Define then $h:=N^{-1}S^jM\in \mathrm{Stab}_{\mathrm{Mod}(T)}(\gamma)$. It follows that $Mh^{-1}=S^{-j}N$ and hence $\mathrm{rot}(Mh^{-1})=-j$. Also, $h\neq I$, as otherwise $S^jM$ would have two factorizations, but by Lemma \ref{lem:SLfact} that's only possible if $S^jM=\langle S\rangle$, contradicting the fact that $\alpha\notin \langle S\rangle \gamma$.

            Conversely, if we start with an element $h\in\mathrm{Stab}_{\mathrm{Mod}(T)}(\gamma)\setminus\{I\}$, we just set $j=-\mathrm{rot}(Mh^{-1})$, and then $S^jMh^{-1}\gamma=S^j\alpha$ implies that $j\in J(\alpha)$.

            For the third step we want to prove two things: first, that $J(\alpha)\subseteq \{0,2\}$, and second that 
            \[2\in J(\alpha)\iff \mathrm{Stab}_{\mathrm{Mod}(T)}(\gamma) \text{ contains a non-trivial torsion element.}\]

            Assume first that $j\in J(\alpha)\setminus \{0,2\}$. By the second step it corresponds to $\mathrm{rot}(Mh^{-1})$ being odd for some $h\in\mathrm{Stab}_{\mathrm{Mod}(T)}(\gamma)$. Note that $Mh^{-1}=M'S^\varepsilon h^{-1}$. By the first step, that corresponds to $S^\varepsilon h^{-1}e_1$ having opposite signs. By the standard position hypothesis, every $h\in \mathrm{Stab}_{\mathrm{Mod}(T)}(\gamma)$ can be written as $h=(-I)^\delta c^k$ with $c\in \{S,U,L\}$, $\delta\in\{0,1\}$, $k\in\mathbb{Z}$. Then, for the following it is enough to assume that $h$ is either a power of $L$, of $-I$, of $S$, or of $U$. In the first three cases, the vector $S^\varepsilon h^{-1}e_1$ can be either $\pm e_1$ or $\pm e_2$. If $h=U^{\pm1}$, then $S^\varepsilon h^{-1}e_1$ can be $(0,1)$, $(-1,0)$, $(-1,-1)$, and $(1,-1)$. From these, only $(1,-1)$ has coordinates of opposite signs, and both $L(1,-1)$ and $R(1,-1)$ no longer do, hence no $M'\neq I$ can produce opposite signs. Hence, $J(\alpha)\subseteq \{0,2\}$.

            Finally, we prove the equivalence from the beginning of the third step. Let us first prove the left implication. A torsion element in the stabilizer can have either order $2,3,4,$ or $6$. Hence, it contains either $-I$ or $U$, or both. Note that if $-I\in\mathrm{Stab}_{\mathrm{Mod}(T)}(\gamma)$, then $M\cdot (-I)=S^2M$ and hence $2\in J(\alpha)$. If $U\in \mathrm{Stab}_{\mathrm{Mod}(T)}(\gamma)$, then note that
            \[
			I+U+U^2=I+\begin{psmallmatrix}
				0 & -1 \\ 1 & -1
			\end{psmallmatrix}+\begin{psmallmatrix}
				0 & -1 \\ 1 & -1
			\end{psmallmatrix}^2=0.
			\]
            Hence, $Me_1+MUe_1+MU^2e_1=0$ are three vectors summing zero. Hence, they cannot be in the same quadrant. Thus, the rotations $\mathrm{rot}(M)$, $\mathrm{rot}(MU)$, and $\mathrm{rot}(MU^2)$ are not all equal. We also know that $\mathrm{rot}(M)=0$ by definition. Therefore, $2\in J(\alpha)$. 

            For the converse direction, if $\mathrm{Stab}_{\mathrm{Mod}(T)}(\gamma)$ is trivial, then $J(\alpha)=\{0\}$. Otherwise, the other elements of $J(\alpha)$ are only generated by the rotation of an infinite-order stabilizer element, which by standard position is $L$. Note that $ML^ne_1=M'S^\varepsilon e_1$ are either positive vectors or on an axis of the positive quadrant. If they are positive, $ML^n$ has rotation $\mathrm{rot}(ML^n)=0$ for all $n\in\ZZ$. Otherwise, forces $M'=L^{n'}$ or $M'=R^{n'}$ for some $n'\in\mathbb{Z}$. Then, since $L$ fixes $\gamma$ and $R$ fixes $S(\gamma)$, we find $\alpha=\langle S\rangle \gamma$, bringing a contradiction. 
        \end{enumerate}

        \end{proof}

	\section{Counting}
	
	In this section, we will prove Theorem \ref{thm:mainthm_subtree}. That is, for a curve $\gamma\in\mathcal{C}(T)$, we will give the counting function in a maximal stable subtree of $T_\gamma$. A maximal stable subtree corresponds to a rooted binary tree with edges corresponding to $L$- and $R$-actions and the root being a stable curve $\llbracket  w \rrbracket$ with $w=\prod_{i=1}^n\mathtt{z}^{e_i}w_i$ such that all $e_i\neq0$ when $n>1$, $w_i$ monotone, and, as usual, $\mathtt{z}=[\mathtt{a},\mathtt{B}]$. Note that $w$ might not be reduced. The existence of a counting function depends on the regularity of the number of cancelling pairs of the word. Define the cancelling pairs number for such a word as: 
	\[
	c(w):=\frac{1}{2}\left(\sum_{i=1}^n\ell(w_i)+4\sum_{i=1}^n |e_i|-\ell(\llbracket w\rrbracket)\right).
	\]
	
	\begin{lemma}\label{lem:cttcanc}
		For a stable word $w=\prod_{i=1}^n\mathtt{z}^{e_i}w_i$, and writing for any $M\in \langle L, R\rangle_+$, $\tilde{M}(w)=\prod_{i=1}^n\mathtt{z}^{e_i}M(w_i)$. Then, 
		\begin{enumerate}
			\item $c(\tilde{M}w)$ is constant for all $M\in \langle L,R \rangle_+\setminus(\langle L\rangle_+ \cup \langle R \rangle_+)$;
			\item  $c(\tilde{M}w)$ is constant for all $M\in\langle L\rangle_+$; and
			\item  $c(\tilde{M}w)$ is constant for all $M\in\langle R\rangle_+$.
		\end{enumerate}
		We will denote these constants by $c_{int}(w)$, $c_L(w)$, and $c_R(w)$, respectively.
	\end{lemma}
	\begin{proof}
		The proof will follow from the following table. Denote $\mathrm{first}(w)$ and $\mathrm{last}(w)$ the first and last letters of a word $w$. Note then how do $L$ and $R$ change the last and first letters of a word.
		\[
		\begin{matrix}
			\mathrm{first}(w)&  \mathrm{first}(L(w)) &    \mathrm{first}(R(w)) \\
			\mathtt{a} &  \mathtt{a}   & \mathtt{a} \\
			\mathtt{b} &  \mathtt{a}   & \mathtt{b} \\
			\mathtt{A} &  \mathtt{A}   & \mathtt{B} \\
			\mathtt{B} &  \mathtt{B}   & \mathtt{B} \\
		\end{matrix}\hspace{50pt} 
		\begin{matrix}
			\mathrm{last}(w) & \mathrm{last}(L(w))  & \mathrm{last}(R(w))  \\
			\mathtt{a} & \mathtt{a} & \mathtt{b} \\
			\mathtt{b} & \mathtt{b} & \mathtt{b} \\
			\mathtt{A} & \mathtt{A} & \mathtt{A} \\
			\mathtt{B} & \mathtt{A} & \mathtt{B} \\
		\end{matrix}
		\]

		All possible cancellations in $\tilde{M}w$ will happen between the monotone subwords $M(w_i)$ and commutators $\mathtt{z,Z}$. Look for example at a subword of the form $w_i z$, where cancellations might happen if $\mathrm{last}(w_i)=\mathtt{A}$. In this case, if $w_i$ is positive, there will be no cancellation in $ M(w_i)\mathtt{z}$ there for any $M\in \langle L, R\rangle_+$, if $\mathrm{last}(w_i)=\mathtt{A}$, there will be one pair cancelling on the entire subtree, and if $\mathrm{last}(w_i)=\mathtt{B}$, there will be no cancellation on the ray $M\in \langle R\rangle_+ $ but one pair cancelling outside.
		
		For every possible cancellation coming from a word touching a commutator, there is one positivity of the word that gives no cancellations on the entire subtree, and one positivity that gives, for one choice of letters a constant one-pair cancellation on the entire subtree, and one choice that gives no cancellation on one ray, and one-pair cancellation on the rest of the tree.
	\end{proof}
	
	Let us now restate Theorem \ref{thm:mainthm_subtree}, and prove it right after. 
	
	\theoremsubtree*
	
	\begin{proof}
		Fix a stable subtree $\Gamma=\langle L, R\rangle_+ \llbracket w\rrbracket$ with stable root $w=\prod_{i=1}^n z^{e_i}w_i$. Note that a monotone word $w_i$ satisfies
		\[
		\ell(L([w_i]))=\ell([w_i])+|w_i|_{\{\mathtt{b},\mathtt{B}\}}\text{, and }\ell(R([w_i]))=\ell([w_i])+|w_i|_{\{\mathtt{a},\mathtt{A}\}}.\] 

        Assume first that $\sum_{i=1}^n|w_i|_{\{\mathtt{b},\mathtt{B}\}},\sum_{i=1}^n|w_i|_{\{\mathtt{a},\mathtt{A}\}}\neq0$. 
		Then, setting $P=\sum_{i=1}^n|w_i|_{\{\mathtt{b},\mathtt{B}\}}$, $Q=\sum_{i=1}^n|w_i|_{\{\mathtt{a},\mathtt{A}\}}$, and $E=\sum_{i=1}^n |e_i|$, it follows from Lemma \ref{lem:cttcanc} that
		\begin{align*}
			\#\{&\alpha\in T \mid \ell(\alpha)=L\}=\varphi_{P,Q}(L-4E+2c_{int}(w))\\&-\#\{x\geq1\mid Px+Q=L-4E+2c_{int}(w)\}-\#\{y\geq1\mid P+Qy=L-4E+2c_{int}(w)\}\\&+\#\{x\geq1\mid Px+Q=L-4E+2c_{L}(w)\}+\#\{y\geq1\mid P+Qy=L-4E+2c_{R}(w)\}.
		\end{align*}	
		The first term counts the number of words with cyclically reduced length $L$ given the linear growth of the monotone words and assuming constant cancellation number with the commutators between them. The periodic correction balances out the possibly different cancellation at the boundary rays of the tree.

        In the case of $P=0$ or $Q=0$, the curve stabilizer contains either $L$ or $R$. In particular, by the standard position hypothesis, this can only happen at the base of the trees $T_\gamma$, $T_{S(\gamma)}$. To avoid the one infinite edge with all nodes representing the same curve, apply the element in $\{L,R\}$ that does not stabilize the curve and we can set the counting from there, where both $P$ and $Q$ coefficients will be non-zero, and at most losing $L_0$ of minimal length in the orbit.
		
	\end{proof}

	\begin{example}\label{ex:counting} Let us compute the counting function for $\gamma=\llbracket \mathtt{a}^2\mathtt{b}^2\mathtt{A}^2\mathtt{b}^2\mathtt{Ab}\rrbracket$. Using the algorithm in Section \ref{sec:algorithm}, we find the stable subtrees and their counting functions shown in Figure \ref{fig:solving_trees}.
		\begin{figure}[ht]\centering
			\begin{overpic}[width=\linewidth]{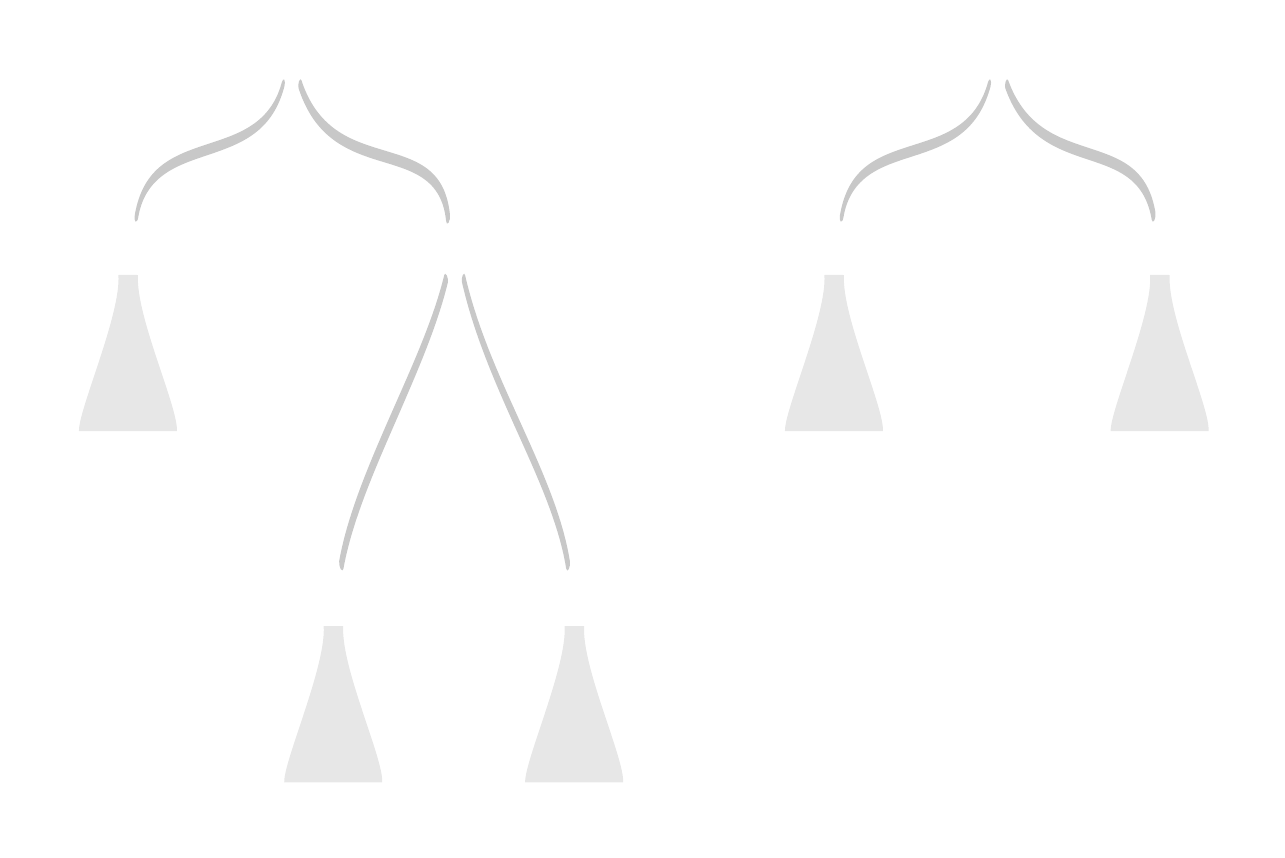}
				\put(15,62){$\llbracket \mathtt{a}^2\mathtt{b}^2\mathtt{A}^2\mathtt{b}^2\mathtt{Ab}\rrbracket$}
				\put(70,62){$\llbracket \mathtt{b}^2\mathtt{A}^2\mathtt{B}^2\mathtt{A}^2\mathtt{BA}\rrbracket$}
				\put(3,46.5){\color{cyan}$\llbracket \mathtt{zB}^2\mathtt{A}^3\mathtt{B}^2\mathtt{AB}\rrbracket$}
				\put(27.5,46.5){$\llbracket \mathtt{BABaZazBzB}\rrbracket$}
				\put(58,46.5){\color{cyan}$\llbracket \mathtt{zbZ(AB)}^2\mathtt{A}^4\rrbracket$}
				\put(82,46.5){\color{cyan}$\llbracket \mathtt{zABAB}^3(\mathtt{AB})^2\mathtt{A}\rrbracket$}
				\put(17,19){\color{cyan}$\llbracket \mathtt{ZazB}^2\mathtt{A}^2\mathtt{BA}^2\mathtt{B}\rrbracket$}
				\put(36,19){\color{cyan}$\llbracket \mathtt{ZbZBAzAzab}^3\rrbracket$}
				\put(6,30){\small$\varphi_{4,5}(L)$} \put(0,27){\small$+\,[0,-1,0,1]\ (\mathrm{mod}\ 4)$}		
				\put(59,30){\small$\varphi_{3,6}(L-4)$}
				\put(86.5,30){\small$\varphi_{5,6}(L)$} \put(76,27){\footnotesize$+\,[0,1,0,0,0,-1]\ (\mathrm{mod}\ 6)$}		
				\put(20,2){\small$\varphi_{4,5}(L-4)$}
				\put(38.5,2){\small$\varphi_{3,5}(L-4)$}
			\end{overpic}\vspace{-15pt}
			\caption{Stable subtrees and counting functions of $\gamma=\llbracket \mathtt{a^2b^2A^2b^2Ab}\rrbracket$.}
			\label{fig:solving_trees}
		\end{figure}
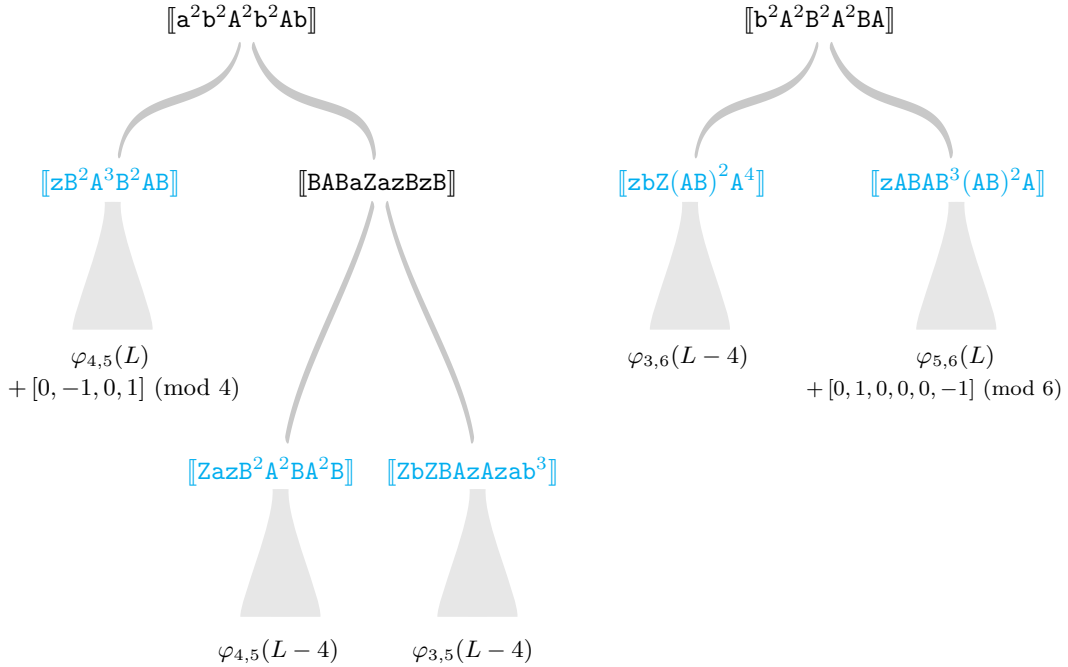
		
		Thus, putting them together, the general count is, for $L\geq12$, 
		\begin{align*}
			\#\{\alpha\in\mathrm{Mod}(T)\cdot \gamma \mid \ell(\alpha)= L\}=\varphi_{3, 5}( L-4) + \varphi_{3, 6}( L-4) + \varphi_{4, 5}( L-4) +\\+ \varphi_{4, 5}( L) + \varphi_{5, 6}( L) + [0, 0, 0, 1, 0, -2, 0, 2, 0, -1, 0, 0]\ (L\ \mathrm{mod}\ 12).
		\end{align*}
		
	\end{example}
	
	Finally, we will prove a trivial bound for the initial length $L_0$ in Theorem \ref{thm:mainthm}. Note however that, generically, it will be closer to the minimal word length in the orbit, and it is computable at every case by following the algorithm in Section \ref{sec:algorithm}.
	
	\begin{lemma}\label{lem:boundL0}
		The constant $L_0$ in Theorem \ref{thm:mainthm} is trivially bounded by 
		\[
		L_0\leq 1+\ell_\mathrm{min}\cdot 2^{1+\ell_\mathrm{min}},
		\]
		where $\ell_\mathrm{min}:=\min\{\ell(\alpha)\mid \alpha\in\mathrm{Mod}(T)\cdot\gamma\}$.
	\end{lemma}
	\begin{proof}
		Let $w_\mathrm{min}$ be a representative in the $\mathrm{Mod}(T)$-orbit with minimal length. Note that the depth of the terminal words in the proof of Proposition \ref{prop:nonstable} is trivially bounded by $1+\ell(\llbracket w_\mathrm{min}\rrbracket)$. Finally, with each $L$- and $R$-action, the length grows at most by doubling.
	\end{proof}
	
	We can also prove the asymptotic growth theorem.
	
	\asymptotic*
\begin{proof}
		Set \[N_{P,Q}(L)=\#\{(x,y)\in\mathbb{Z}_{>0}^2\mid Px+Qy=L\}.\]
		It is well-known (see e.g. \cite[Theorem 6]{BR02}) that, when $\gcd(P,Q)\mid L$, then
		$N_{P,Q}(L)=\frac{L\gcd(P,Q)}{PQ}+E(L)$ and vanishes otherwise. Here $E$ is a function bounded by $1$. Write now by considering the parameter $d$ as the greatest common divisor of a solution,
		\[N_{P,Q}(L)=\sum_{d\mid L}\varphi_{P,Q}(L/d).\]
		Hence, by Möbius inversion formula, with $\mu$ being the Möbius function,
		\begin{align*}\varphi_{P,Q}(L)&=\sum_{d\mid L}\mu(d)N_{P,Q}(L/d)\\&=\sum_{d\mid (L/\gcd(P,Q))}\mu(d)N_{(P,Q)}(L/d)\\&=\sum_{d\mid (L/\gcd(P,Q))}\mu(d)\left(\frac{\gcd(P,Q)}{PQ}\frac{L}{d}+E\left(\frac{L}{d}\right)\right)\\&=\varphi_{1,1}\left(\frac{L}{\gcd(P,Q)}\right)\frac{\gcd(P,Q)^2}{PQ}+\sum_{d\mid (L/\gcd(P,Q))}\mu(d)E\left(\frac{L}{d}\right),
        \end{align*}        
		where the second equality comes from the solutions only existing when $L/d$ is a multiple of $\gcd(P,Q)$ and the third equality follows from the Möbius inversion formula again. Given that, by Dirichlet's Divisor Problem,
        \[
        \sum_{l=1}^L\sum_{d\mid (l/\gcd(P,Q))}\mu(d)E\left(\frac{l}{d}\right)\leq\sum_{l=1}^L\sum_{d\mid (l/\gcd(P,Q))}|\mu(d)|=O(L\log L),
        \]
        the theorem follows from the known asymptotic $\sum_{l=1}^L\varphi_{1,1}(l)\sim3L^2/\pi^2$.
	\end{proof}

	\section{Intersection resolutions}\label{sec:resolutions}
	
	The following is a standard procedure. For a longer background see \cite{Turaev1991}. Let $\gamma\in\mathcal{M}(T)$. Take a representative $c:(S^1)^k\to T$---i.e. a \emph{multiloop}---in \emph{minimal position}: minimizing the number of intersection points and with all intersections being simple. Fix $p\in\mathrm{im}(c)$ an intersection point. We consider the two local smoothings that produce multiloops with lower intersection, as in Figure \ref{fig:resolution1}.
	
	\begin{figure}[h!]
		\centering
		\includegraphics[width=0.8\linewidth]{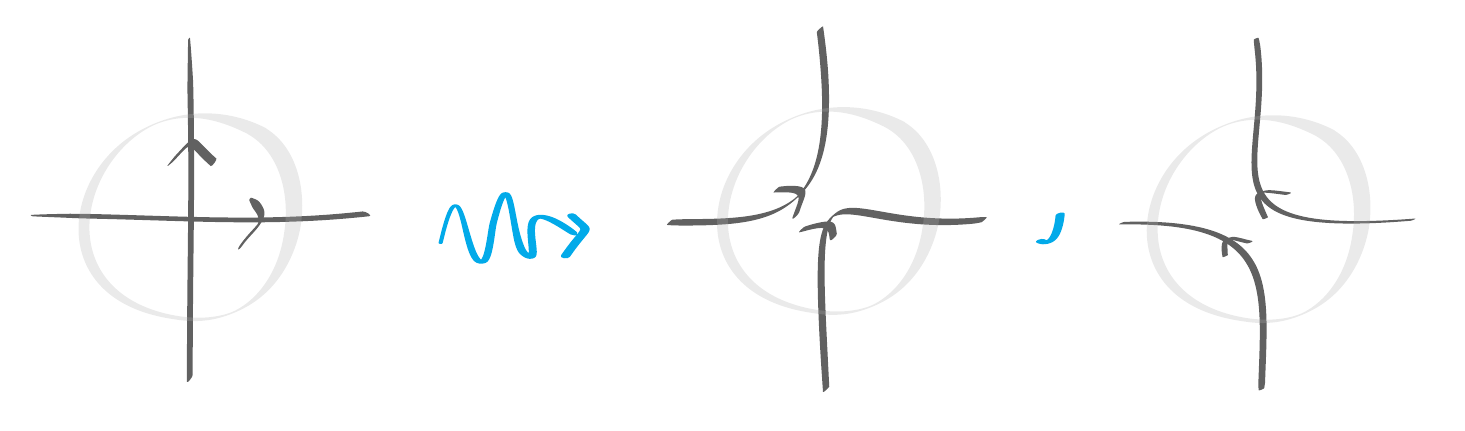}
		\caption{The two local surgeries resolving an intersection.}
		\label{fig:resolution1}
	\end{figure}
	
	For a pair of a multiloop in minimal position and a self-intersection point $(c,p)$ with $c:(S^1)^k\to T$, define the corresponding intersection point of $\gamma=[c]$ as $([p],\gamma)=\{(\eta(p,1),\eta(c,1))\mid \eta:  (S^1)^k\times [0,1] \to T \text{ homotopy with }\eta(s,0)=c(s) \}$.
	
	This allows us to extend the two operations on the multiloop to its free homotopy class. We will write them with implicit dependence on the intersection point:
	\[r_{op},r_{or}:\gamma \to [r_{op}(c,p)]\sqcup [r_{or}(c,p)]\subseteq \mathcal{M}(T).\]
	
	Note that the number of components of the multicurve is not fixed under these operations. It can both increase, decrease, and stay invariant: for example, among the resolutions of the self-intersection of $\{\mathtt{a}^2\mathtt{b}^2\}$, one increases and one leaves invariant the number of components; on the other hand, both resolutions for $\{\mathtt{a},\mathtt{b}\}$  decrease the number of components.
	
	The above names come from orientation-preserving and orientation-reversing, since the action is as follows. If the self-intersection is inside one component, one resolution splits the component into two respecting any given orientation, and the other one reattaches them reversing the orientation of one of the components.  If the self-intersection is between two components, one resolution attaches the two components at a basepoint with coherent orientation, and the other resolution attaches them in reverse orientation.
	
	\begin{figure}[h!]
		\centering
		\begin{overpic}[width=0.95\linewidth]{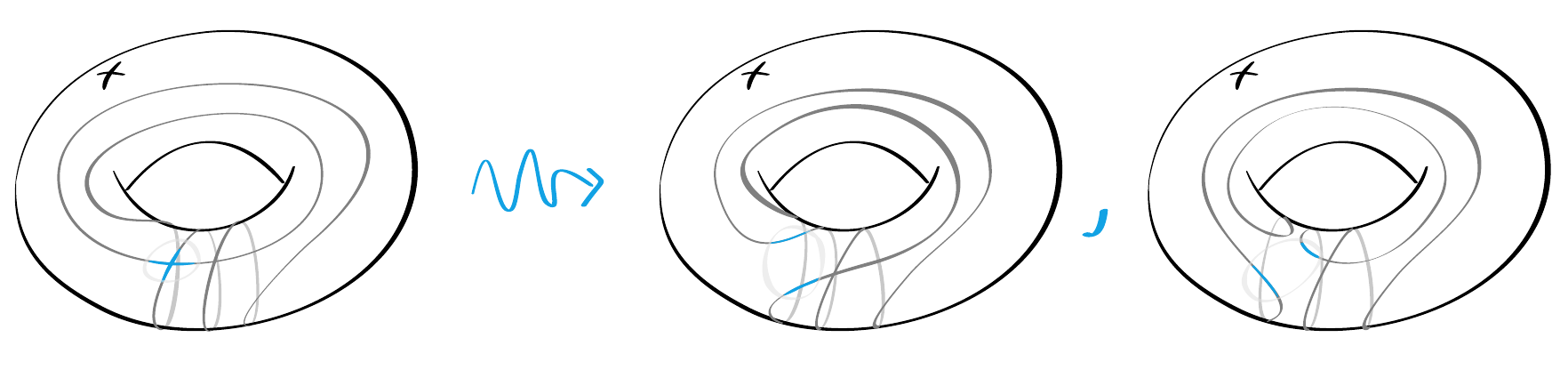}
			\put(18,0){$\{\llbracket \mathtt{a}^3\mathtt{b}^2\rrbracket\}$}
			\put(55,0){$\{\llbracket \mathtt{ab}\rrbracket,\llbracket \mathtt{a}^2\mathtt{b}\rrbracket\}$}
			\put(90,0){$\{\llbracket \mathtt{a}^2\mathtt{bAB}\rrbracket\}$}
		\end{overpic}
		\caption{The two resolution maps on a self-intersection.}
		\label{fig:resolution2}
	\end{figure}
	
	Now extend the definition of the intersection point to a $\mathrm{Mod}(T)$ orbit. For a given multicurve $\gamma\in\mathcal{M}(T)$ and self-intersection point $[p]$, given $f\in\mathrm{Mod}(T)$, define $f(\gamma,[p])=[(\tilde f\circ c, \tilde{f}(p))]$, where $\tilde{f}\in\mathrm{Homeo}^+(T)$ is a representative of $f$ and $(c,p)$ is a representative of $(\gamma,[p])$.
	
	\begin{lemma}\label{lem:equiv}
		Let $(\gamma,[p])$ be a multicurve and an intersection point and $r$ a fixed resolution of the pair. Then, $r$ is equivariant with respect the $\mathrm{Mod}(T)$-action. That is, for any $f\in\mathrm{Mod}(T)$, 
		\[
		f(r(\gamma,[p]))=r(f(\gamma,[p])).
		\]
	\end{lemma}
	
	\begin{proof}
		The main ingredient is that every resolution move is a Whitehead-type move on a train-track.
		Fix a multicurve $\gamma$.
		
		Let $\tau$ be the train-track induced on the surface by a minimal representative of $\gamma$. Let $\tau'$ be the train-track induced on the surface by a minimal representative of $r(\gamma,p)$. 
		The sequence of Whitehead-type moves from $\tau'$ to $\tau$ is a central splitting in the notation of \cite[Section~3.13]{Mos03}, for which $\tau$ carries $\tau'$ and hence there is a carrying map $(T,\tau')\to (T,\tau)$ from $\tau'$ into $\tau$ \cite[Section~3.5]{Mos03}. This map is homotopic to the identity.
	\end{proof}
	
	Now, by the above lemma, the resolution maps can be extended to the entire mapping class orbit of any multicurve $\gamma$:
	
	\[r_{op},r_{or}:\mathrm{Mod}(T)\cdot \gamma\to \mathrm{Mod}(T)\cdot r_{op}(\gamma)\cup \mathrm{Mod}(T)\cdot r_{or}(\gamma).\]
	
	\paragraph{Linking pairs} For a modern reference on intersection and linking pairs follow \cite{Cha2004}. The following procedure goes back to Cohen and Lustig in \cite{CL87}. We will refer to it as the Cohen-Lustig algorithm for finding intersections. Define the following lexicographic ordering. Start with the cyclic ordering $\mathtt{a}<\mathtt{b}<\mathtt{A}<\mathtt{B}$. In order to compare two words $\omega,\delta$ in $\{\mathtt{a},\mathtt{b},\mathtt{A},\mathtt{B}\}$, write $\omega=\omega_1\cdots\omega_n$ and $\delta=\delta_1\cdots \delta_m$ with $\omega_i$ and $\delta_i$ letters. Consider the infinite powers by concatenation $\omega^\infty,\delta^\infty$. We say that $\omega<\delta$ if either $\omega_1<\delta_1$, or $\omega_i=\delta_i$ for all $i=1,\hdots, k-1$ for some $k>1$ and $\omega_k<\delta_k$ at the cyclic shift of the initial ordering such that it starts with $\omega_{k-1}^{-1}$.
	
	Let $\omega=\omega_1\cdots \omega_n$ be a word. Denote by $\omega^i=\omega_{i}\omega_{i+1}\cdots \omega_{i-1}$ the cyclic shift of $\omega$. A \emph{linking pair} inside the word is a pair of integers $(i,j)_\omega$ such that, up to cyclic shift and taking inverses at one word,
	\[
	\omega^i<\omega^j<{(\omega^i)}^{-1}<{(\omega^j)}^{-1}.
	\]
	We can similarly define a linking pair between two different words $\omega,\delta$. Write $\omega=\omega_1\cdots\omega_n$ and $\delta=\delta_1\cdots\delta_n$. A \emph{linking pair} between these words is a pair of integers $(i,j)_{\omega,\delta}$ such that, up to cyclic shift and taking inverses at one word,
	\[
	\omega^i<\delta^j<(\omega^i)^{-1}<(\delta^{j})^{-1}.
	\]
	
	Now consider equivalence classes of linking pairs by $(i,j)_{\omega,\delta}\sim (i+1,j+1)_{\omega,\delta}$ if $\omega_i=\delta_j$ and $(i+1,j)_{\omega,\delta}\sim (i,j+1)_{\omega,\delta}$ if $\omega_i=\delta_j^{-1}$, and an analogous definition if it is between the same word, by exchanging $\delta$ for $\omega$. By abuse of notation we will call equivalence classes of linking pairs also linking pairs, and, if not saying the contrary explicitly, we will be always referring to equivalence classes.
	
	Also abusing notation, we may refer to a linking pair as a pair of words $(\omega,\delta)$ whenever $(1,1)_{\omega,\delta}$ is a linking pair.
	
	\paragraph{Combinatorics of resolutions} It is well-known that, for primitive curves, equivalence classes of linking pairs correspond to intersection points (self-intersection if considering only one word, intersection if considering two), of the curves represented by the words (see e.g. \cite{DL17}). This correspondence can be seen as, given an intersection point between two curves, e.g. $\alpha,\beta\in\mathcal{C}(T)$, fix the intersection point as a basepoint in the fundamental group. The representatives of $\alpha$ and $\beta$ found in this way correspond to a linking pair $(u,v)$ algebraically with $\llbracket u \rrbracket=\alpha$ and $\llbracket v \rrbracket=\beta$. On the other hand, the two resolutions of this intersection point correspond to concatenating $\alpha$ and $\beta$ with the same orientation or with inverse orientations at the given basepoints, i.e. the $u$ and $v$ words in represent $\alpha$ and $\beta$ with the given basepoint, then the two resolutions are $\llbracket uv\rrbracket,\llbracket uv^{-1}\rrbracket\in\mathcal{C}(T)$, as concatenations of words.
	
	Analogously, a self-intersection point corresponds to a linking pair given by the two permutations of the word representing the curve via fixing the basepoint at the intersection. In a similar fashion, if the two permutations of the word are $\omega_1\cdots \omega_n\in\llbracket \omega\rrbracket$ and $\omega_i\cdots \omega_{i-1}\in\llbracket \omega\rrbracket$, then we can write the curve as $uv$, with $u=\omega_1\cdots \omega_{i-1}$ and $v=\omega_i\cdots \omega_n$, and the two resolutions correspond to $\{u,v\},\{uv^{-1}\}\in\mathcal{M}(T)$.
	
	Finally, if a curve is not primitive, the orientation-preserving resolution corresponds to splitting the curve separating some of its powers, and the orientation-reversing to cancelling out some of its powers. Note that this is coherent with the self-intersection resolutions.
	
	\begin{lemma}\label{lem:max}
		Given $(\gamma,[p])$ a multicurve and an intersection point, the two possible resolutions $r_1,r_2$ satisfy
		\[
		\ell(\gamma)=\max\{\ell(r_1(\gamma,p)), \ell(r_2(\gamma,p))\}.
		\]
	\end{lemma}
	
	\begin{proof}
		Resolving an intersection falls into one of the following cases. 
		\begin{enumerate}
			\item If the self-intersection is in a unique component $\gamma=\llbracket w\rrbracket$, where $w$ is a cyclically reduced word, then the two possible resolutions act as follows: there is a split of the word into two subwords $w=uv$ such that the orientation-preserving resolution transforms $\gamma$ into the multicurve $\{\llbracket u\rrbracket,\llbracket v\rrbracket\}$, and the non-orientation-preserving resolution transforms $\gamma$ into a curve $\llbracket uv^{-1}\rrbracket$. The lemma in this case follows from: if $uv$ is a cyclically reduced word, then either both $u$ and $v$ are cyclically reduced, or $uv^{-1}$ is cyclically reduced. This is a standard fact in combinatorics, let us prove it. Assume without loss of generality that $u$ is not cyclically reduced, hence, $u=xu'x^{-1}$ for some letter $x$, and write $v=y_1v'y_2$ for some letters $y_1,y_2$. Then $uv^{-1}=xu'x^{-1}y_2^{-1}v'^{-1}y_1^{-1}$ is not cyclically reduced if and only if either $x=y_1$ or $x=y_2$, which in both cases implies that $w=xu'x^{-1}y_1v'y_2$ is not cyclically reduced, finding a contradiction.
			
			\item If the intersection is between two components $\{\llbracket u\rrbracket,\llbracket v\rrbracket\}$ of the multicurve, $u,v$ being cyclically reduced words, the resolution of the intersection maps the multicurve to either $\llbracket uv\rrbracket$ or $\llbracket uv^{-1}\rrbracket$. Once again, the lemma follows by elementary combinatorics. If two words $u,v$ are cyclically reduced, so is either $uv$ or $uv^{-1}$. We will prove it now. 
			Assume $uv$ is not reduced. Assume without loss of generality that the cancellation appears at the border, i.e. $u=xu'$ and $v=v'x^{-1}$ for some letter $x$. Then $uv^{-1}=xu'xv'^{-1}$, which is cyclically reduced because both $u$ and $v$ are.
		\end{enumerate}
		The above shows that the intersection always has one resolution corresponding to reordering the letters of a representative with no cancellations appearing, and hence the resolution does not lose any length.
	\end{proof}
	
	We will need one last ingredient to prove the resolution regularity from Theorem \ref{thm:resolutions}. We have to go back to the coefficients from Lemma \ref{lem:cttcanc}, and prove an extra property.
	
	\begin{lemma}\label{lem:ceven} Let $w$ be a stable word. In the notation of Lemma \ref{lem:cttcanc}, the constant $c_{int}$ is even. 
	\end{lemma}
	\begin{proof}
		Look at the table at the proof of Lemma \ref{lem:cttcanc}. It follows directly that for $c_{int}$, at the situations $\mathtt{z}w\mathtt{z}$ and $\mathtt{Z}w\mathtt{Z}$ there is a contribution of $2$ or of $0$ depending on the positivity of $w$. On the other hand, at the situations $\mathtt{z}w\mathtt{Z}$ and $\mathtt{Z}w\mathtt{z}$ there is a contribution of $1$ independently of the positivity, but these situations will happen an even number of times.
	\end{proof}
	
	Let us finally restate Theorem \ref{thm:resolutions} and prove it.
	
	\resolutions*
	
	\begin{proof}
		Let us first recall the construction of $\mathcal{R}$. Starting at an essential curve $\gamma\in\mathcal{C}(T)$, we choose a self-intersection point and resolve it in the two ways, giving rise to two multicurves. We iteratively choose new intersection points and resolve them until we get at most $2^{i(\gamma)}$ resolutions to simple multicurves. We call a \emph{complete list of full resolutions} of $\gamma$ a finite set $\mathcal{R}$ obtained in this way in which, for any $r\in \mathcal{R}$, $r:\mathrm{Mod}(T)\cdot \gamma \to \mathrm{Mod}(T)\cdot r(\gamma)$ is defined equivariantly---i.e. $r(f(\gamma))=f(r(\gamma))$ for any $f\in\mathrm{Mod}(T)$. 
		Fix such a complete list of full resolutions $\mathcal{R}$. Note that $\mathcal{R}$ might depend on the choices of which intersections to resolve at each step.
		Denote these resolutions by $\mathcal{R}=\{r_1,\hdots, r_N:\mathrm{Mod}(T)\cdot \gamma\to \mathcal{M}(T)\}$.
		
		Note that for any curve $f(\gamma)\in \mathrm{Mod}(T)$, each resolution is a simple multicurve, hence is of the form $r_i(\gamma)=\{\alpha,\overset{(n_\alpha)}{\hdots},\alpha, [\mathtt{a},\mathtt{b}],\overset{(n_c)}{\hdots},[\mathtt{a},\mathtt{b}]\}$. Then we know the growths \[\ell(r_i(L(\gamma)))=\ell(r_i(\gamma))+n_\alpha|\alpha|_{\mathtt{b}},\]
		and, 
		\[
		\ell(r_i(R(\gamma)))=\ell(r_i(\gamma))+n_\alpha|\alpha|_{\mathtt{a}}.
		\]
		Now, using Lemma \ref{lem:cttcanc}, we can also predict that, if there was a resolution in the list that preserves length on all the interior of a stable subtree, what the multicurve should be at every node so that the growth coincides.\\
		
		More concretely, by the same argument as in the proof of Theorem \ref{thm:mainthm_subtree}, if the root of a stable subtree is $w=\prod_{i=1}^n\mathtt{z}^{e_i}w_i$, then the only possible length-preserving resolution $r\in\{r_1,\hdots, r_N\}$ at the entire interior of the tree has to be such that $r(w)=\{\alpha,\overset{(n_\alpha)}{\hdots},\alpha, [\mathtt{a},\mathtt{b}],\overset{(n_c)}{\hdots},[\mathtt{a},\mathtt{b}]\}$  with
		\begin{enumerate}
			\item 
			
			$|\alpha|_\mathtt{a}\cdot n_\alpha=\sum_{i=1}^n|w_i|_{\mathtt{a},\mathtt{A}}$;
			\item $|\alpha|_\mathtt{b}\cdot n_\alpha=\sum_{i=1}^n|w_i|_{\mathtt{b},\mathtt{B}}$;
			\item $n_\alpha=\gcd(\sum_{i=1}^n|w_i|_{\mathtt{a},\mathtt{A}},\sum_{i=1}^n|w_i|_{\mathtt{b},\mathtt{B}})$; and
			\item $n_c=\sum_{i=1}^n|e_i|-c_{int}/2$.
		\end{enumerate}
		
		We will now prove that there is always such a resolution in $\mathcal{R}$.
		
		The claim follows from finiteness of the list of resolution. Note that by the Lemma \ref{lem:max}, at every curve of the tree at least one resolution of the list is length-preserving. By finiteness of the list, one of them has to be length-preserving on infinitely-many curves in the stable subtree. Such a resolution has to have the same growth as the curves, whose growth is known---see the proof of Theorem \ref{thm:mainthm_subtree}. 
		
		Hence that is a resolution, and moreover, since its growth is the same as the original curve in the stable subtree, it is length-preserving on the entire interior of the stable subtree.
		
		Finally, maximality follows directly by construction. The root of the stable subtree ensures that the growth changes as soon as we apply $L^{-1}$ or $R^{-1}$ on the tree, hence it is not the predicted length which coincides with the resolution growth, and hence the resolution is not length-preserving at a bigger connected subtree.
	\end{proof}

	The above proof gives a very descriptive definition of the coefficients $(P,Q,K)$ at every stable subtree.
	
	\begin{corollary}\label{cor:coeff} Let $\gamma\in\mathcal{C}(T)$ be the root of one stable subtree. Let $r(\gamma)=\{\alpha,\overset{(n_\alpha)}{\hdots},\alpha, [\mathtt{a},\mathtt{b}],$ $,\overset{(n_c)}{\hdots},[\mathtt{a},\mathtt{b}]\}$ be its predominant length-preserving full resolution, where $\alpha$ is a simple curve. Then, in the notation of Theorem \ref{thm:mainthm_subtree},
		\[
		P=n_\alpha|\alpha|_\mathtt{b}, \, Q=n_\alpha|\alpha|_\mathtt{a},\,\text{and } K=n_c.
		\]
	\end{corollary}
	
	\begin{proof}
		It follows straight from the action of $L$ and $R$, and Lemma~\ref{lem:equiv}. 
	\end{proof}
	
	In the notation of the proof of Theorem \ref{thm:resolutions}, note that the same growth argument from Theorem \ref{thm:resolutions} ensures that there is a constant length-preserving resolution at the $\langle L \rangle\gamma$ and the $\langle R \rangle\gamma$ rays from the root of a stable subtree. That resolution's image at the root is less direct to compute. For $\langle L \rangle \gamma$, writing $\gamma=\llbracket w\rrbracket$ with $w=\prod_{i=1}^n \mathtt{z}^e_i w_i$ stable,
	\begin{enumerate}
		\item $n_\alpha\cdot |\alpha|_\mathtt{b}=\sum |w_i|_{\mathtt{b},\mathtt{B}}$;
		\item $\gcd(|\alpha|_\mathtt{a},|\alpha|_\mathtt{b})=1$; and
		\item  $\ell(\gamma)=n_\alpha\cdot (|\alpha|_\mathtt{a}+|\alpha|_\mathtt{b})+4n_c$.
	\end{enumerate}
	The stable resolution is the only $r(\gamma)=\{\alpha,\overset{(n_\alpha)}{\hdots},\alpha, [\mathtt{a},\mathtt{b}],\overset{(n_c)}{\hdots},[\mathtt{a},\mathtt{b}]\}$ satisfying the conditions above. For $\langle R \rangle \gamma$ the first equation changes to $n_\alpha\cdot |\alpha|_\mathtt{a}=\sum |w_i|_{\mathtt{a},\mathtt{A}}$ and the rest stay the same.
	
	\begin{example} Continuing Example \ref{ex:counting}, we can find the length-preserving resolutions $r_1,\hdots ,r_5:\mathrm{Mod}(T)\cdot\gamma \to \mathcal{M}(T)$ at the interior of each stable subtree from Theorem \ref{thm:resolutions} as in Figure \ref{fig:resolving_trees}.
		\begin{figure}[ht]
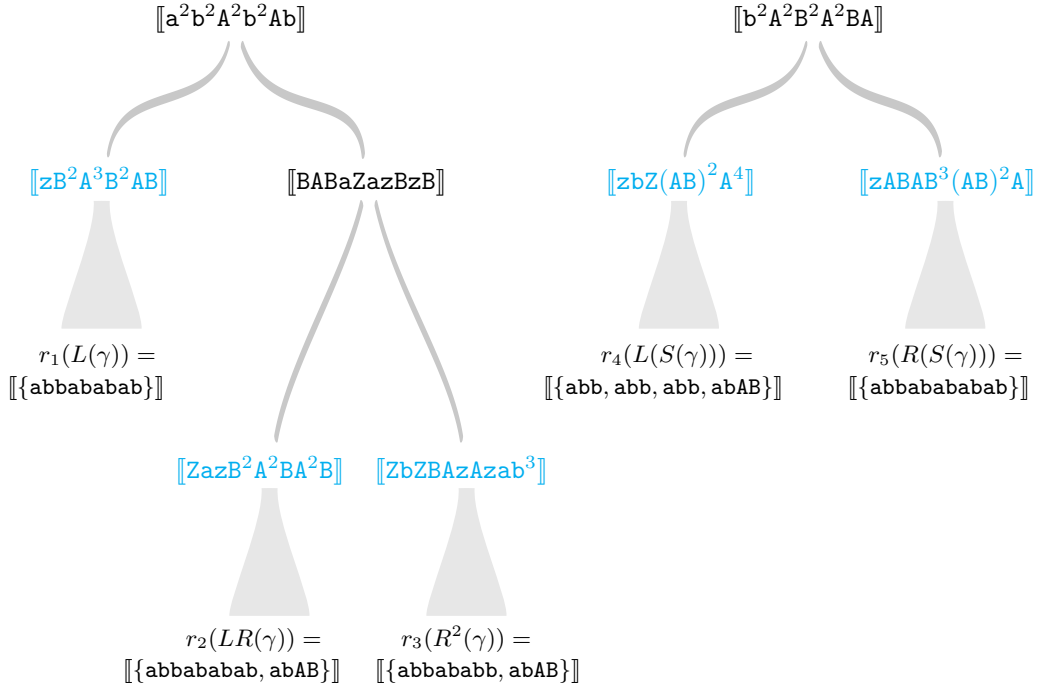
\centering
			\begin{overpic}[width=\linewidth]{exampletr.pdf}
				\put(15,62){$\llbracket \mathtt{a}^2\mathtt{b}^2\mathtt{A}^2\mathtt{b}^2\mathtt{Ab}\rrbracket$}
				\put(70,62){$\llbracket \mathtt{b}^2\mathtt{A}^2\mathtt{B}^2\mathtt{A}^2\mathtt{BA}\rrbracket$}
				\put(3,46.5){\color{cyan}$\llbracket \mathtt{zB}^2\mathtt{A}^3\mathtt{B}^2\mathtt{AB}\rrbracket$}
				\put(27.5,46.5){$\llbracket \mathtt{BABaZazBzB}\rrbracket$}
				\put(58,46.5){\color{cyan}$\llbracket \mathtt{zbZ(AB)}^2\mathtt{A}^4\rrbracket$}
				\put(82,46.5){\color{cyan}$\llbracket \mathtt{zABAB}^3(\mathtt{AB})^2\mathtt{A}\rrbracket$}
				\put(17,19){\color{cyan}$\llbracket \mathtt{ZazB}^2\mathtt{A}^2\mathtt{BA}^2\mathtt{B}\rrbracket$}
				\put(36,19){\color{cyan}$\llbracket \mathtt{ZbZBAzAzab}^3\rrbracket$}
				\put(4,30){\small$r_1(L(\gamma))=$} \put(1,27){\small$\llbracket \{\mathtt{abbababab}\}\rrbracket$}
				\put(57.5,30){\small$r_4(L(S(\gamma)))=$}
				\put(52,27){\small$\llbracket \{\mathtt{abb},\mathtt{abb},\mathtt{abb},\mathtt{abAB}\}\rrbracket$}
				\put(83,30){\small$r_5(R(S(\gamma)))=$} \put(81,27){\small$\llbracket \{\mathtt{abbabababab}\}\rrbracket$}		
				\put(18,3){\small$r_2(LR(\gamma))=$}
				\put(12,0){\small$\llbracket \{\mathtt{abbababab},\mathtt{abAB}\}\rrbracket$}
				\put(38.5,3){\small$r_3(R^2(\gamma))=$}
				\put(36,0){\small$\llbracket \{\mathtt{abbababb},\mathtt{abAB}\}\rrbracket$}
				
			\end{overpic}\vspace{-8pt}
			\caption{Stable subtrees and interior length-preserving resolutions of $\gamma=\llbracket \mathtt{a^2b^2A^2b^2Ab}\rrbracket$.}
			\label{fig:resolving_trees}
		\end{figure}

	\end{example}
	
	\section{Algorithm}\label{sec:algorithm}
	
	The following algorithm starts from $w$ a cyclic word in $\{\mathtt{a},\mathtt{b},\mathtt{A},\mathtt{B}\}$ and, by using rewriting rules, finds representatives in stable form in finitely-many steps for any $L,R$-path. Then, at every infinite subtree, the counting function might be computed following the proof of Theorem \ref{thm:mainthm_subtree}.
	
	\paragraph{Rewrite rules.}
	We have three types of rewrite rules.
	These are rules that change a word without changing the underlying element of $F_2$.
	All rules are interpreted as replacing a given subword of a cyclic word with another word. The rules (C1)$,\hdots,$(C4) are cancellation moves, the rules (F1),$\hdots,$(F4) are full commutator moves (for all $k\geq 1$), and the rules (P1),$\hdots,$(P4) are partial commutator moves (for all $k\geq 1$).
	
	\[
	\begin{matrix}
		\text{(C1)} & \mathtt{aA}\to\emptyset &\hspace{20pt} & \text{(F1)} & \mathtt{aB}^k \mathtt{A b}\to (\mathtt{zB})^{k-1}\mathtt{z}&\hspace{20pt} & \text{(P1)} & \mathtt{aB}^k \mathtt{A}\to (\mathtt{zB})^k \\
		\text{(C2)} & \mathtt{Aa}\to\emptyset &\hspace{20pt} & \text{(F2)} & \mathtt{Bab}^k \mathtt{A }\to \mathtt{Z}(\mathtt{bZ})^{k-1}&\hspace{20pt} & \text{(P2)} & \mathtt{ab}^k\mathtt{A}\to (\mathtt{bZ})^k \\
		\text{(C3)} & \mathtt{bB}\to\emptyset &\hspace{20pt} & \text{(F3)} &\mathtt{Ba}^k \mathtt{bA}\to (\mathtt{Za})^{k-1}\mathtt{Z} &\hspace{20pt} & \text{(P3)} & \mathtt{Ba}^k \mathtt{b}\to (\mathtt{Z a})^k \\
		\text{(C4)} & \mathtt{Bb}\to\emptyset &\hspace{20pt} & \text{(F4)} & \mathtt{aBA}^k \mathtt{b}\to (\mathtt{zA})^{k-1}\mathtt{z}&\hspace{20pt} & \text{(P4)} & \mathtt{BA}^k \mathtt{b}\to (\mathtt{A z})^k. 
	\end{matrix}
	\]
	
	Moreover, for a cyclically reduced cyclic word $w$, let $L^*(w)$ be the cyclic reduction of $L(w)$, and $R^*(w)$ be the cyclic reduction of $R(w)$.
	Then $R^*(w)$ operates by transforming digrams and monograms according to the following table, where digrams take precedence over monograms.
	\[
	\begin{matrix}
		\mathtt{aB}&\mathtt{bA}&\mathtt{a}&\mathtt{b}&\mathtt{A}&\mathtt{B}&\mathtt{z}&\mathtt{Z}\\
		\downarrow&\downarrow&\downarrow&\downarrow&\downarrow&\downarrow&\downarrow&\downarrow\\
		\mathtt{a}&\mathtt{A}&\mathtt{ab}&\mathtt{b}&\mathtt{BA}&\mathtt{B}&\mathtt{z}&\mathtt{Z}
	\end{matrix}
	\]
	The word resulting from these substitutions is cyclically reduced, and is precisely $R^*(w)$.
	A similar table exists for $L^*(w)$.

	Recall that in our notation $\mathtt{z}=[\mathtt{a},\mathtt{B}]=\mathtt{aBAb}$ and $\mathtt{Z}=\mathtt{B abA}$ and that these two commutators are fixed by $L$ and $R$ in $F_2$. However, during the algorithm, the commutator letters are considered symbolically different by the rewrite rules.
	
	Recall also that a curve is stable if it can be written as $\prod_{i=1}^n\mathtt{z}^{e_i}w_i$ for $e_i\neq0$ and $w_i$ non-empty monotone words.  For the algorithm, we will introduce a weaker version of stable curves. We call \emph{terminal} word a (cyclic) word that does not (cyclically) contain any of the following subwords: $\mathtt{aA}$, $\mathtt{aB}$, $\mathtt{Aa}$, $\mathtt{Ab}$, $\mathtt{bA}$, $\mathtt{bB}$, $\mathtt{Ba}$, $\mathtt{Bb}$. The algorithm is as follows.
	
	\paragraph{The algorithm:} Start with a cyclically reduced word $w$, and with an empty list (a multiset) of terminal words $\mathcal{L}$.
	\begin{enumerate}
		\item\label{step:1} Apply all possible full commutator moves. Call the new word $w'$.
		\item Apply all possible partial commutator moves. Call the new word $w''$.
		\item If $w''$ is terminal, save it in the list $\mathcal{L}$.
		
		Otherwise, compute $L(w'')$ and apply all possible cancellation moves. Call it $L^*(w'')$. Go back to Step \ref{step:1} with the new word $L^*(w'')$. Do the same for $R$.
		\item Start again Step \ref{step:1} with $S(w)$.
        \item For every $w\in\mathcal L$: if $L(\llbracket w \rrbracket)=\llbracket w \rrbracket$, substitute $w\in\mathcal L$ by $R(w)$. Do symmetrically if $R$ stabilizes instead.
        \item Reduce the set such that no curve is represented twice by words in $\mathcal{L}$.
		\item Compute the formula for each word in $\mathcal{L}$. 
	\end{enumerate}
	
	\paragraph{The formula:} For a terminal word $w$, following the proof of Theorem \ref{thm:mainthm_subtree}, the counting formula at the subtree rooted there can be computed as follows. It will follow from Proposition \ref{prop:alg} that the word will be written in the form $w=\prod_{i=1}^n\mathtt{z}^{e_i}w_i$ with $w_i$ non-empty monotone words. Set $P_w=\sum_{i=1}^n|w_i|_{\{\mathtt{b},\mathtt{B}\}}$, $Q_w=\sum_{i=1}^n|w_i|_{\{\mathtt{a},\mathtt{A}\}}$, and $E_w=\sum_{i=1}^n |e_i|$. Set also $c_{int}(w)$, $c_{L}(w)$ and $c_{R}(w)$ as in Lemma \ref{lem:cttcanc}. Then, define
	{\small\begin{align*}&\Phi_w(L):=\varphi_{P_w,Q_w}(L-4E_w+2c_{int}(w))\\&-[0,\overset{(4E_w-2c_{int}(w))}{\hdots},0,1,0,\hdots,0]\ (L\ \mathrm{mod} \ P_w)+[0,\overset{(4E_w-2c_{R}(w))}{\hdots},0,1,0,\hdots,0]\ (L\ \mathrm{mod} \ P_w) \\
			&-[0,\overset{(4E_w-2c_{int}(w))}{\hdots},0,1,0,\hdots,0]\ (L\ \mathrm{mod} \ Q_w)+[0,\overset{(4E_w-2c_{L}(w))}{\hdots},0,1,0,\hdots,0]\ (L\ \mathrm{mod} \ Q_w).
	\end{align*}}
	Set finally $m:=2$ if $\gamma$ has non-trivial torsion elements in the stabilizer, and $m:=4$ otherwise. Then, Theorem \ref{thm:mainthm} can be written as follows. Let $w$ be a cyclically reduced word in $\{\mathtt{a},\mathtt{b},\mathtt{A},\mathtt{B}\}$ in standard position such that $\ell(\llbracket w\rrbracket)$ is minimal in $\mathrm{Mod}(T)\cdot \llbracket w\rrbracket$. Let $\mathcal{L}$ the list of terminal words coming from the algorithm. Then, 
	\[
	\#\{\alpha\in\mathrm{Mod}(T)\cdot \llbracket w\rrbracket\mid \ell(\alpha)=L\}=m\sum_{v\in\mathcal{L}}\Phi_v(L),
	\]
	for all $L\geq L_0$ with $L_0=\max_{v\in\mathcal{L}}(\{\ell(\llbracket v\rrbracket)\}\cup\{\ell(\llbracket w\rrbracket)+1\})$.

	\hfill
	
	We will now prove that the algorithm works.
	\begin{proposition}\label{prop:alg} For any cyclic word $w$,
		\begin{enumerate}
			\item \label{prop:algorithm}the algorithm terminates in finite time,
			\item \label{nopanic} a terminal word contains no $\mathtt{zZ},\mathtt{Zz}$ subwords.
		\end{enumerate}
	\end{proposition}
	
	We will start by proving the termination. Note that the second point is necessary because it implies that terminal words for the algorithm are actually stable words, and hence, the formula can be computed by following the proof of Theorem \ref{thm:mainthm_subtree}.
	
	\subsection{Termination}
	
	Let $w$ be a cyclically reduced cyclic word.
	For an index $i$, let
	\begin{align*}
		\lambda(w,i)&:=\max\{k\ge 1\mid \forall\, 0\le j<k\quad w_{i-j}=w_i\}\text{, and}\\
		\rho(w,i)&:=\max\{k\ge 1\mid \forall\, 0\le j<k\quad w_{i+j}=w_i\}.
	\end{align*}
	
	Define the following quantities:
	\begin{align*}
		\mu_L(w)&:=\sum\{\rho(w,i+1)\mid w_iw_{i+1}=\mathtt{B a}\}+\sum\{\lambda(w,i)\mid w_iw_{i+1}=\mathtt{A b}\},\\
		\mu_R(w)&:=\sum\{\rho(w,i+1)\mid w_iw_{i+1}=\mathtt{aB}\}+\sum\{\lambda(w,i)\mid w_iw_{i+1}=\mathtt{bA}\}\text{, and}\\
		\mu(w)&:=\mu_L(w)+\mu_R(w).
	\end{align*}

	\begin{lemma}
		Let $w$ be a cyclically reduced cyclic word, and let $u$ be the result of applying a full commutator move to $w$.
		Then $\mu(u)<\mu(w)$.
	\end{lemma}
	\begin{proof}
		Let us only do move (F1), as the rest are symmetric.
		Write $w=\mathtt{aB}^k\mathtt{Ab}v$ for some freely reduced word $v$, so that $u=(\mathtt{zB})^{k-1}\mathtt{z}v$.
		Note that $\mu(u)=\mu(v)$ since $\mu_L(u)=\mu_L(v)$ and $\mu_R(u)=\mu_R(v)$.
		However, $\mu_L(w)\ge\mu_L(v)+k$ because of the $\mathtt{aB}^k$ subword, so $\mu_L(w)>\mu_L(u)$ and hence $\mu(w)>\mu(u)$.
	\end{proof}
	
	\begin{lemma}
		Let $w$ be a cyclically reduced cyclic word that does not admit any full commutator moves, and let $u$ be the result of applying a partial commutator move to $w$.
		Then $\mu(u)<\mu(w)$.
	\end{lemma}
	\begin{proof}
		Let us only do move (P1), as the rest are symmetric.
		Write $w=\mathtt{aB}^k\mathtt{A} v$ for some freely reduced word $v$, so that $u=(\mathtt{zB})^k v$.
		Note that $v$ does not start with $\mathtt{a}$ or $\mathtt{b}$.
		Then one checks that $\mu_L(u)=\mu_L(v)$ and $\mu_R(u)=\mu_R(v)$.
		However, $\mu_L(w)\ge\mu_L(v)+k$ because of the $\mathtt{aB}^k$ subword, so $\mu(w)>\mu(u)$.
	\end{proof}
	
	A cyclically reduced cyclic word $w$ is a \emph{branch cyclic word} if it is not terminal and it admits no (full or partial) commutator moves.
	
	\begin{proposition}
		Let $w$ be a cyclically reduced branch cyclic word.
		Then the following hold:
		\begin{enumerate}
			\item $\mu_R(R^*(w))=\mu_R(w)-\#\{i\mid w_iw_{i+1}\in\{\mathtt{aB},\mathtt{bA}\}\}$;
			\item $\mu_L(R^*(w))\le\mu_L(w)$;
			\item Let $R^*(w)''$ be $R^*(w)$ after applying all possible (full and then partial) commutator moves, then $\mu(R^*(w)'')<\mu(w)$.
		\end{enumerate}
	\end{proposition}
	\begin{proof}
		Let $u=R^*(w)$.
		\begin{enumerate}
			\item From the table for $R^*$, a subword of $u=R^*(w)$ of the form $\mathtt{aB}^k$ (with $k$ maximal for the $\mathtt{B}$-run) can only arise in two ways: from a subword $(\mathtt{aB})\mathtt{B}^k$ of $w$ or from a subword $(\mathtt{aB})\mathtt{B}^{k-1}\mathtt{A}$ of $w$. The second option is excluded because it would imply a possible (F1) or (P1) commutator move. Hence, every maximal $\mathtt{aB}^k$ block comes from an $\mathtt{aB}^{k+2}$ block. Symmetrically, every maximal $\mathtt{b}^k\mathtt{a}$ block comes from a $\mathtt{b}^{k+1}\mathtt{a}$ block. 
			
			Moreover, note by the table of $R^*$ that $R^*$ creates no $\mathtt{aB}$ or $\mathtt{bA}$ diagram.
			
			Hence, $R^*$ decreases the size of all $\mathtt{aB}^k$ and $\mathtt{b}^k\mathtt{A}$ blocks by $1$ and creates no new ones, which implies the equality in the statement.
			
			\item From the table for $R^*$, we see that any subword of $u$ of the form $\mathtt{B a}^k$ must come from a subword of $w$ of the form $\mathtt{B}(\mathtt{aB})^{k-1}\mathtt{a}$, that is, $(\mathtt{B a})^k$.
			A subword of $w$ of the form $(\mathtt{B a})^k$ already contributes at least $k$ to $\mu_L(w)$.
			Similarly, a subword $\mathtt{A}^k \mathtt{b}$ of $u$ must come from a subword $(\mathtt{A b})^k$ of $w$, which again contributes at least $k$ to $\mu_L(w)$, proving the statement.
			\item Let $v$ be $u$ after applying all (full and then partial) commutator moves.
			If $u$ admits at least one commutator move, then $\mu(v)<\mu(u)\le\mu(w)$, so we are done; otherwise, $v=u$.
			Suppose for a contradiction that $\mu(u)\ge\mu(w)$.
			Then, by parts 1 and 2, we have that $\mu_R(u)=\mu_R(w)$ and $\mu_L(u)=\mu_L(w)$, implying that $w$ has no $\mathtt{aB}$ or $\mathtt{bA}$ subwords.
			However, note that $R^*(\mathtt{A b})=\mathtt{BAb} $ and $R^*(\mathtt{B a})=\mathtt{B ab}$, both of which give rise to commutator moves.
			But we know that $u$ admits no commutator moves, so $w$ has no $\mathtt{A b}$ or $\mathtt{B a}$ subwords.
			In other words, $w$ is terminal, contradicting the assumption. 
		\end{enumerate}
	\end{proof}
	
	A perfectly symmetric statement (with perfectly symmetric proof) holds for $L^*$.
	
	\begin{proposition}
		Let $w$ be a cyclically reduced branch cyclic word.
		Then the following hold:
		\begin{enumerate}
			\item $\mu_R(L^*(w))\le\mu_R(w)$;
			\item $\mu_L(L^*(w))=\mu_L(w)-\#\{i:w_iw_{i+1}\in\{\mathtt{A b},\mathtt{B a}\}\}$;
			\item Let $L^*(w)''$ be $L^*(w)$ after applying all possible (full and then partial) commutator moves, then $\mu(L^*(w)'')<\mu(w)$.
		\end{enumerate}
	\end{proposition}
	
	Hence, Point \ref{prop:algorithm} of Proposition \ref{prop:alg} follows from the strict decrease of the complexity $\mu$. Moreover,
	
	\begin{corollary}
		The maximum depth of the recursion of the algorithm is at most $\mu(w)\le\ell(\llbracket w \rrbracket)$.
	\end{corollary}
	
	\subsection{No commutators cancel}
	
	Point \ref{nopanic} of Proposition \ref{prop:alg} is a direct corollary of the following proposition. Call a cyclic word \emph{allowed} if it does not have any of the following subwords: $\mathtt{zZ},\mathtt{Z z},\mathtt{zB a},\mathtt{Z aB}, \mathtt{bA z},\mathtt{A bZ}$. Call it \emph{forbidden} otherwise.
	
	\begin{proposition}
		Let $w$ be a cyclically reduced cyclic word in $\{\mathtt{a},\mathtt{b},\mathtt{z},\mathtt{A},\mathtt{B},\mathtt{Z}\}$.
		If $w$ is allowed, then the following claims hold.
		\begin{enumerate}
			\item If $w$ admits a full commutator move $m$, then $m(w)$ is allowed;
			\item If $w$ admits no full commutator moves but admits a partial commutator move $m$, then $m(w)$ is allowed;
			\item If $w$ is a branch word, then $L^*(w)$ and $R^*(w)$ are allowed.
		\end{enumerate}
	\end{proposition}
	\begin{proof}\leavevmode We will prove it point by point.
		\begin{enumerate}
			\item Let us do only (F1), since the other full commutator moves are symmetric.
			Let $u=m(w)$.
			Write $w=\mathtt{aB}^k\mathtt{A b} v$ for some reduced word $v$, so that $u=(\mathtt{zB})^{k-1}\mathtt{z}v$.
			Since $w$ is allowed, $v$ does not start or end with $\mathtt{Z}$, so $u$ has no $\mathtt{zZ}$ or $\mathtt{Z z}$ subwords.
			Since $w$ is cyclically reduced, it does not start with $\mathtt{B}$ or end with $\mathtt{A}$.
			In particular, $u$ does not have $\mathtt{zB a}$ or $\mathtt{bA z}$ subwords (since $v$ also does not).
			We deduce that $u$ is allowed.
			\item Let us do only (P1), since the other partial commutator moves are symmetric.
			Let $u=m(w)$.
			Write $w=\mathtt{aB}^k\mathtt{A} v$ for some reduced word $v$, so that $u=(\mathtt{zB})^k v$.
			Because $w$ admits no commutator moves and is reduced, we know that $v$ does not start with $\mathtt{a}$ or $\mathtt{b}$, and it does not end with $\mathtt{Z}$ or $\mathtt{A}$. We check then that none of the forbidden subwords are in $u$.
			
			For $\mathtt{zZ}$, and $\mathtt{Zz}$: the block $(\mathtt{zB})^k$ contains no $\mathtt{Z}$ and $v$ does not begin with $\mathtt{Z}$. For $\mathtt{zBA}$: within $(\mathtt{zB})^k$ each $\mathtt{z}$ is followed by $\mathtt{B}$ and then $\mathtt{z}$, never by $\mathtt{a}$, and at the junction it is excluded because $v$ cannot start with $\mathtt{a}$. For $\mathtt{bAz}$: $u$ contains a $\mathtt{z}$ either inside $(\mathtt{zB})^k$ or inside $v$, since the letter before each $(\mathtt{z}B)$-block is either $B$ or it is the boundary, in which case it cannot be $\mathtt{bA}$ because $w$ is cyclically reduced. For $\mathtt{ZaB}$ and $\mathtt{AbZ}$: these involve $\mathtt{Z}$, which happens only inside $v$, and then it is controlled by $w$ being allowed.

			\item Let us do only $R^*$, since $L^*$ is symmetric.
			Let $u=R^*(w)$.
			This is again a case check using the table.
			It is clear that $u$ has no $\mathtt{zZ}$ or $\mathtt{Z z}$ subwords.
			A $\mathtt{zB a}$ subword of $u$ can only come from a $\mathtt{zB a}$ of $w$ (forbidden).
			A $\mathtt{Z aB}$ subword can only come from $\mathtt{Z aB}$ (forbidden).
			A $\mathtt{bA z}$ subword can only come from $\mathtt{bA z}$ (forbidden).
			An $\mathtt{A bZ}$ subword can only come from $\mathtt{A bZ}$ (forbidden).
			So $u$ is allowed. \qedhere
		\end{enumerate}
	\end{proof}
	
	Hence, if $w$ is a cyclically reduced cyclic word in $\{\mathtt{a},\mathtt{b},\mathtt{A},\mathtt{B}\}$.
	Then every terminal word derived from $w$ through the formula algorithm has no $\mathtt{Z z}$ or $\mathtt{zZ}$ subwords.

	\bibliographystyle{alpha}
	\bibliography{references}
	
	\printaddresses

\end{document}